\documentclass[a4paper,dvipsnames]{amsart}

\usepackage[english]{babel}
\usepackage[T1]{fontenc}

\usepackage{amsmath}
\usepackage{amssymb}
\usepackage{amsaddr}
\usepackage{enumitem}
\setlist[enumerate]{label=\textnormal{(\arabic*)}}
\usepackage{amsthm}
\usepackage[dvipsnames]{xcolor}
\usepackage[pagebackref,hypertexnames=false, colorlinks, citecolor=BrickRed,linkcolor=MidnightBlue, urlcolor=BrickRed]{hyperref}
\usepackage[capitalise]{cleveref}
\crefname{condition}{condition}{conditions}
\crefname{formula}{formula}{formulas}

\usepackage{nicefrac}

\usepackage{tikz}
\usepackage{tikz-cd}
\usetikzlibrary{plotmarks}

\usepackage{partmac}

\usepackage{dsfont}
\usepackage{mathtools}
\usepackage{scalerel}
\usepackage{stackengine}

\usepackage{etoolbox}
\apptocmd{\sloppy}{\hbadness 10000\relax}{}{}
\apptocmd{\sloppy}{\vbadness 10000\relax}{}{}

\usepackage{mathtools}

\usepackage[a4paper,top=3cm,bottom=2cm,left=3cm,right=3cm,marginparwidth=1.75cm]{geometry}

\usepackage{comment}  % for commenting out rather big writings 

\usepackage[english]{babel}
\usepackage{colonequals} % For a better font for := 

\usepackage{hyphenat}

\usepackage{academicons}
\definecolor{orcidcol}{HTML}{A6CE39}

\usepackage{faktor}

\colorlet{Malte}{OliveGreen}
\colorlet{Philipp}{blue}

\renewcommand{\restriction}{\mathbin{\upharpoonright}}

\newcommand{\gbullet}{{\textcolor{gray}{\bullet}}}

\newcommand{\unitalsqcup}{\ensuremath\mathbin{\ooalign{$\sqcup$\crcr\hidewidth\tiny\raisebox{.2em}{$1$}\hidewidth}}}

\newcommand{\bigunitalsqcup}{\mathop{\ooalign{$\bigsqcup$\crcr\hidewidth\smaller\raisebox{-.05em}{$1$}\hidewidth}}}

\newcommand{\orcidlogo}{{\includegraphics[width=\fontcharht\font`l]{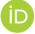}}}

\newcommand{\twofacedprod}[3][.7]{
  \setstackgap{L}{#1\baselineskip}\mathbin{\scriptsize{{% \raisebox{.7em}
    {
      \Vectorstack{{#2} {#3}}}}}}
\setstackgap{L}{\baselineskip}}

\newcommand{\ufprod}[2]{\mathop{{}_{#1}{\overrightarrow\Asterisk}_{#2} }}

\newcommand{\uttprod}[4]{
  \twofacedprod{ {}_{#1}\otimes_{#2}} {{}_{#3}\otimes_{#4}}
}

\newcommand{\uffprod}[4]{\twofacedprod[.9]{
    \ufprod{#1}{#2}}{\ufprod{#3}{#4}}
}
\newcommand{\ubprod}[2]{\mathop{{}_{#1}{\overleftarrow\Asterisk}_{#2} }} %deformed right free product

\newcommand{\ubfprod}[4]{\twofacedprod[.9]{
    \ufprod{#1}{#2}}{\ubprod{#3}{#4}}
 } 
\providecommand{\Asterisk}{\mathbin{\scalerel*{\ast}{\otimes}}}

\newcommand{\bool}{\mathbin{
    \scalerel*{\diamond}{\otimes}}   
  }

\newcommand{\boolboolprod}{\twofacedprod{\bool}{\bool}}

\def\Pg (#1,#2){
	\pgfpathcircle{\pgfpointxy{#1}{#2}}{\Pcolw em}
	\pgfsetfillcolor{lightgray}
	\pgfusepath{fill}
}

\def\Pgbullet(#1,#2){
  \pgfpathcircle{\pgfpointxy{#1}{#2}}{\Pcolw em}
  \pgfsetfillcolor{stroke,gray}
  \pgfusepath{fill}
}

\def\Pgs (#1,#2){
	\pgftransformshift{\pgfpointxy{#1}{#2}}
	\pgfpathmoveto{\pgfpoint{-\Pcolw em}{-\Pcolw em}}
	\pgfpathlineto{\pgfpoint{-\Pcolw em}{\Pcolw em}}
	\pgfpathlineto{\pgfpoint{\Pcolw em}{\Pcolw em}}
	\pgfpathlineto{\pgfpoint{\Pcolw em}{-\Pcolw em}}
	\pgfpathclose
	\pgfsetfillcolor{lightgray}
	\pgfusepath{fill}
	\pgftransformreset
}

\newcommand{\wsquare}{\qcol}
\newcommand{\bsquare}{\Qcol}

\newcommand{\Pbs}{\PQ}
\newcommand{\Pws}{\Pq}

\newcommand{\partxyoyoyxo}{\Partition{
    \Pblock 0.2 to 1.2:1,5
    \Pblock 0.2 to 0.8:2,3,4
    \Ppoint0.2 \Pb:2,3,5
    \Ppoint0.2 \Pw:1,4 }
}

\newcommand{\partyozxoyox}{\Partition{
    \Psingletons 0.2 to 0.6:2
    \Pblock 0.2 to 1.2:3,5
    \Pblock 0.2 to 0.9:1,4
    \Ppoint0.2 \Pb:1,3,4
    \Ppoint0.2 \Pw:2,5 }
}

\newcommand{\partxxoxoyxo}{\Partition{
    \Psingletons 0.2 to 0.8:4
    \Pblock 0.2 to 1.2:1,3,2,5
    \Ppoint0.2 \Pb:2,3,5
    \Ppoint0.2 \Pw:1,4 }
}

\newcommand{\partxyoyxo}{\Partition{
    \Pblock 0.2 to 1.2:1,4
    \Pblock 0.2 to 0.8:2,3
    \Ppoint0.2 \Pb:2,4
    \Ppoint0.2 \Pw:1,3 }
}

\newcommand{\partxyoxoxyo}{\Partition{
    \Pblock 0.2 to 1.2:1,3,4
    \Pblock 0.2 to 0.8:2,5
    \Ppoint0.2 \Pb:2,3,5
    \Ppoint0.2 \Pw:1,4 }
}

\newcommand{\partxyoxozyo}{\Partition{
    \Psingletons 0.2 to 0.6:4
    \Pblock 0.2 to 1.2:1,3
    \Pblock 0.2 to 0.9:2,5
    \Ppoint0.2 \Pb:2,3,5
    \Ppoint0.2 \Pw:1,4 }
}

\newcommand{\partxyoxoxzo}{\Partition{
    \Psingletons 0.2 to 0.8:2,5
    \Pblock 0.2 to 1.2:1,3,4
    \Ppoint0.2 \Pb:2,3,5
    \Ppoint0.2 \Pw:1,4 }
}

\newcommand{\partxyoxozwo}{\Partition{
    \Psingletons 0.2 to 0.8:2,4,5
    \Pblock 0.2 to 1.2:1,3
    \Ppoint0.2 \Pb:2,3,5
    \Ppoint0.2 \Pw:1,4 }
}

\newcommand{\nestbullet}{\Partition{ \Psingletons 0.2 to 1.2:1,3 \Psingletons 0.2 to 0.8:2 \Pline
    (1,1.2) (3,1.2) \Ppoint0.2 \Pg:1,3 \Ppoint0.2 \Pb:2 }}

\newcommand{\nestcirc}{\Partition{ \Psingletons 0.2 to 1.2:1,3 \Psingletons 0.2 to 0.8:2 \Pline
    (1,1.2) (3,1.2) \Ppoint0.2 \Pg:1,3 \Ppoint0.2 \Pw:2 }}

\newcommand{\nestcircbullet}{\Partition{ \Psingletons
    0.2 to 1.2:1,4 \Psingletons 0.2 to 0.8:2,3 \Pline (1,1.2) (4,1.2)
    \Ppoint0.2 \Pg:1,4 \Pline (2,0.8) (3,0.8) \Ppoint0.2 \Pw:2 \Ppoint0.2
    \Pb:3 }}

\newcommand{\crosscircbullet}{\Partition{
    \Psingletons 0.2 to 0.8:1,3
    \Psingletons 0.2 to 1.2:2,4
    \Pline (2,1.2) (4,1.2)
    \Pline (1,0.8) (3,0.8)
    \Ppoint0.2 \Pw:2
    \Ppoint0.2 \Pb:3
    \Ppoint0.2 \Pg:1,4
  }}

\newcommand{\crosscirccirc}{\Partition{
    \Psingletons
    0.2 to 0.8:1,3 \Psingletons 0.2 to 1.2:2,4 \Pline (2,1.2) (4,1.2) \Pline
    (1,0.8) (3,0.8) \Ppoint0.2 \Pw:2,3 \Ppoint0.2 \Pg:1,4
  }}
\newcommand{\crossbulletbullet}{\Partition{
    \Psingletons 0.2 to 0.8:1,3 \Psingletons 0.2 to 1.2:2,4 \Pline (2,1.2)
    (4,1.2) \Pline (1,0.8) (3,0.8) \Ppoint0.2 \Pb:2,3 \Ppoint0.2 \Pg:1,4
  }}

\newcommand{\sqpartxoyo}{\Partition{
    \Psingletons 0.2 to 1.2:1,2
    \Ppoint0.2 \Pbs:1,2 }
}

\newcommand{\sqpartxy}{\Partition{
    \Psingletons 0.2 to 1.2:1,2
    \Ppoint0.2 \Pws:1,2 }
}

\newcommand{\sqpartxyo}{\Partition{
    \Psingletons 0.2 to 1.2:1,2
    \Ppoint0.2 \Pbs:2 
    \Ppoint0.2 \Pws:1 }
}

\newcommand{\sqpartxyzyozox}{\Partition{
    \Pblock 0.2 to 1.2:1,6
    \Pblock 0.2 to 0.6:2,4
    \Pblock 0.2 to 0.9:3,5
    \Ppoint0.2 \Pgs:1,6
    \Ppoint0.2 \Pbs:4,5
    \Ppoint0.2 \Pws:2,3 }
}

\newcommand{\sqpartxyzxozy}{\Partition{
    \Pblock 0.2 to 1.2:1,4
    \Pblock 0.2 to 0.9:2,6
    \Pblock 0.2 to 0.6:3,5
    \Ppoint0.2 \Pgs:1,6
    \Ppoint0.2 \Pbs:4
    \Ppoint0.2 \Pws:2,3,5 }
}
\newcommand{\sqpartxyxyox}{\Partition{
    \Pblock 0.2 to 1.2:1,3,5
    \Pblock 0.2 to 0.8:2,4
    \Ppoint0.2 \Pgs:1,5
    \Ppoint0.2 \Pbs:4
    \Ppoint0.2 \Pws:2,3 }
}

\newcommand{\sqpartxyyoyx}{\Partition{
    \Pblock 0.2 to 1.2:1,5
    \Pblock 0.2 to 0.8:2,3,4
    \Ppoint0.2 \Pgs:1,5
    \Ppoint0.2 \Pbs:3
    \Ppoint0.2 \Pws:2,4 }
}

\newcommand{\sqpartxyyozozx}{\Partition{
    \Pblock 0.2 to 1.2:1,6
    \Pblock 0.2 to 0.8:2,3
    \Pblock 0.2 to 0.8:4,5
    \Ppoint0.2 \Pgs:1,6
    \Ppoint0.2 \Pbs:3,4
    \Ppoint0.2 \Pws:2,5 }
}
\newcommand{\sqpartxyxoyx}{\Partition{
    \Pblock 0.2 to 1.2:1,3,5
    \Pblock 0.2 to 0.8:2,4
    \Ppoint0.2 \Pgs:1,5
    \Ppoint0.2 \Pbs:3
    \Ppoint0.2 \Pws:2,4 }
}
\newcommand{\sqpartxyzzoyx}{\Partition{
    \Pblock 0.2 to 1.2:1,6
    \Pblock 0.2 to 0.9:2,5
    \Pblock 0.2 to 0.6:3,4
    \Ppoint0.2 \Pgs:1,6
    \Ppoint0.2 \Pbs:4
    \Ppoint0.2 \Pws:2,3,5 }
}

\newcommand{\sqpartxyzyzox}{\Partition{
    \Pblock 0.2 to 1.2:1,6
    \Pblock 0.2 to 0.6:2,4
    \Pblock 0.2 to 0.9:3,5
    \Ppoint0.2 \Pgs:1,6
    \Ppoint0.2 \Pbs:5
    \Ppoint0.2 \Pws:2,3,4 }
}

\newcommand{\sqpartxyxzyoz}{\Partition{
    \Pblock 0.2 to 1.2:1,3
    \Pblock 0.2 to 1.2:4,6
    \Pblock 0.2 to 0.8:2,5
    \Ppoint0.2 \Pgs:1,6
    \Ppoint0.2 \Pbs:5
    \Ppoint0.2 \Pws:2,3,4 }
}

\newcommand{\sqpartxyxozoyz}{\Partition{
    \Pblock 0.2 to 1.2:1,3
    \Pblock 0.2 to 1.2:4,6
    \Pblock 0.2 to 0.8:2,5
    \Ppoint0.2 \Pgs:1,6
    \Ppoint0.2 \Pbs:3,4
    \Ppoint0.2 \Pws:2,5 }
}

\newcommand{\sqnestbullet}{\Partition{ \Psingletons 0.2 to 1.2:1,3 \Psingletons 0.2 to 0.8:2 \Pline
    (1,1.2) (3,1.2) \Ppoint0.2 \Pgs:1,3 \Ppoint0.2 \Pbs:2 }}

\newcommand{\sqnestcirc}{\Partition{ \Psingletons 0.2 to 1.2:1,3 \Psingletons 0.2 to 0.8:2 \Pline
    (1,1.2) (3,1.2) \Ppoint0.2 \Pgs:1,3 \Ppoint0.2 \Pws:2 }}

\newcommand{\sqnestcircbullet}{\Partition{ \Psingletons
    0.2 to 1.2:1,4 \Psingletons 0.2 to 0.8:2,3 \Pline (1,1.2) (4,1.2)
    \Ppoint0.2 \Pgs:1,4 \Pline (2,0.8) (3,0.8) \Ppoint0.2 \Pws:2 \Ppoint0.2
    \Pbs:3 }}

\newcommand{\sqcrosscircbullet}{\Partition{
    \Psingletons 0.2 to 0.8:1,3
    \Psingletons 0.2 to 1.2:2,4
    \Pline (2,1.2) (4,1.2)
    \Pline (1,0.8) (3,0.8)
    \Ppoint0.2 \Pws:2
    \Ppoint0.2 \Pbs:3
    \Ppoint0.2 \Pgs:1,4
  }}

\newcommand{\sqcrosscirccirc}{\Partition{
    \Psingletons
    0.2 to 0.8:1,3 \Psingletons 0.2 to 1.2:2,4 \Pline (2,1.2) (4,1.2) \Pline
    (1,0.8) (3,0.8) \Ppoint0.2 \Pws:2,3 \Ppoint0.2 \Pgs:1,4
  }}
\newcommand{\sqcrossbulletbullet}{\Partition{
    \Psingletons 0.2 to 0.8:1,3 \Psingletons 0.2 to 1.2:2,4 \Pline (2,1.2)
    (4,1.2) \Pline (1,0.8) (3,0.8) \Ppoint0.2 \Pbs:2,3 \Ppoint0.2 \Pgs:1,4
  }}

\newtheorem{theorem}{Theorem}[section]
\newtheorem{lemma}[theorem]{Lemma}
\newtheorem{corollary}[theorem]{Corollary}

\theoremstyle{definition} % sentences appear in roman style (not italic)
\newtheorem{definition}[theorem]{Definition}
\newtheorem{remark}[theorem]{Remark}
\newtheorem{example}[theorem]{Example}
\newtheorem{observation}[theorem]{Observation}
\newtheorem{notation}[theorem]{Notation}

\theoremstyle{remark}

\title[Classification of multi-faced independences: combinatorial approach]{Towards a classification of multi-faced independences: \texorpdfstring{\\}{} a combinatorial approach}

\thanks{*The work of both authors was supported by German Research Foundation (DFG) grant no.\ 397960675. The work of MG was carried out as a postodoctoral researcher at Saarland University, during the tenure of an ERCIM `Alain Bensoussan' Fellowship Programme at NTNU Trondheim, as a guest researcher at Saarland University in the scope of the SFB-TRR 195, and as a postdoctoral scientific employee at University of Greifswald. The work of PV was partially carried out as a PhD student and scientific employee at University of Greifswald.
}

\author{Malte Gerhold$^{\MakeLowercase{a},*}$ \& Philipp Var\v{s}o$^{\MakeLowercase{b},*}$ }

\address{$^a$\href{https://orcid.org/0000-0003-4029-1108}{ 
   \orcidlogo\,0000-0003-4029-1108}\\Saarland University, Fachbereich Mathematik}

\address{$^b$\href{https://orcid.org/0000-0001-9199-2516}{ 
   \orcidlogo\,0000-0001-9199-2516}}

\makeatletter
\hypersetup{
pdftitle={\@title},
pdfauthor={Malte Gerhold and Philipp Var\v{s}o},
}
\makeatother
\date{}

\newsavebox{\hasseset}
\begin{lrbox}{\hasseset}
    \begin{tikzcd}[column sep=tiny, row sep=normal]
        &
        &
        \begin{smallmatrix}
            \mathrm{A_{\circ\bullet}}
            \\
            =
            \\
            \bigl\langle
            \nestcircbullet,
            \crosscircbullet
            \bigr\rangle= \bigl\langle
            \crosscirccirc,
            \crosscircbullet
            \bigr\rangle
        \end{smallmatrix}
        &
        &
        \\
        &
        \begin{smallmatrix}
            \mathrm{NC_{\circ\bullet}}
            \\
            =
            \\
            \bigl\langle
            \nestcircbullet
            \bigr\rangle
        \end{smallmatrix}
        \arrow[ru, dash]
        &
        \begin{smallmatrix}
            \mathrm{pC_{\circ\bullet}}
            \\
            =
            \\
            \bigl\langle
            \crosscirccirc,
            \crossbulletbullet
            \bigr\rangle
        \end{smallmatrix}
        \arrow[u, dash]
        &
        \begin{smallmatrix}
            \mathrm{biNC_{\circ\bullet}}
            \\
            =
            \\
            \bigl\langle
            \Partition{
                \Psingletons 0.2 to 0.8:1,3 \Psingletons 0.2 to 1.2:2,4 \Pline (2,1.2) (4,1.2) \Pline
                (1,0.8) (3,0.8) \Ppoint0.2 \Pw:2 \Ppoint0.2 \Pb:3 \Ppoint0.2
                \Pg:1,4}
            \bigr\rangle
        \end{smallmatrix}
        \arrow[lu, dash]
        &
        \\
        \begin{smallmatrix}
            \mathrm{NC_\circ A_\bullet}
            \\
            =
            \\
            \bigl\langle
            \nestcirc,
            \crossbulletbullet
            \bigr\rangle
        \end{smallmatrix}
        \arrow[rru, dash]
        &
        &
        &
        &
        \begin{smallmatrix}
            \mathrm{A_\circ NC_\bullet}
            \\
            =
            \\
            \bigl\langle
            \nestbullet,
            \crosscirccirc
            \bigr\rangle
        \end{smallmatrix}
        \arrow[llu, dash]
        \\
        \begin{smallmatrix}
            \mathrm{I_\circ A_\bullet}
            \\
            =
            \\
            \bigl\langle
            \crossbulletbullet
            \bigr\rangle
        \end{smallmatrix}
        \arrow[u, dash]
        &
        &
        \begin{smallmatrix}
            \mathrm{pNC_{\circ \bullet}}
            \\
            =
            \\
            \bigl\langle
            \nestcirc,
            \nestbullet
            \bigr\rangle
        \end{smallmatrix}
        \arrow[llu, dash]
        \arrow[rru, dash]
        \arrow[uu, dash]
        \arrow[luu, crossing over, dash]
        \arrow[ruu, crossing over, dash]
        &
        &
        \begin{smallmatrix}
            \mathrm{A_\circ I_\bullet}
            \\
            =
            \\
            \bigl\langle
            \crosscirccirc
            \bigr\rangle
        \end{smallmatrix}
        \arrow[u, dash]
        \\
        \begin{smallmatrix}
            \mathrm{I_\circ NC_\bullet}
            \\
            =
            \\
            \bigl\langle
            \nestbullet
            \bigr\rangle
        \end{smallmatrix}
        \arrow[rru, dash]
        \arrow[u, dash]
        &
        &
        &
        &
        \begin{smallmatrix}
            \mathrm{NC_\circ I_\bullet}
            \\
            =
            \\
            \bigl\langle
            \nestcirc
            \bigr\rangle
        \end{smallmatrix}
        \arrow[llu, dash] \arrow[u, dash]
        \\
        &
        &
        \begin{smallmatrix}
            \mathrm{I_{\circ\bullet}}
            \\
            =
            \\
            \bigl\langle
            \bigr\rangle
        \end{smallmatrix}
        \arrow[llu, dash] \arrow[rru, dash]
        &
        &
    \end{tikzcd}
\end{lrbox}

\begin{document}

\begin{abstract}
  We determine a set of necessary conditions on a partition-indexed family of complex numbers to be the ``highest coefficients'' of a positive and symmetric multi-faced universal product; i.e.\ the product associated with a multi-faced version of noncommutative stochastic independence, such as bifreeness. The highest coefficients of a universal product are the weights of the moment-cumulant relation for its associated independence. We show that these conditions are \emph{almost} sufficient, in the sense that whenever the conditions are satisfied, one can associate a (automatically unique) symmetric universal product with the prescribed highest coefficients.  Furthermore, we give a quite explicit description of such families of coefficients, thereby producing a list of candidates that must contain all positive symmetric universal products. We discover in this way four (three up to trivial face-swapping) previously unknown moment-cumulant relations that give rise to symmetric universal products; to decide whether they are positive, and thus give rise to independences which can be used in an operator algebraic framework, remains an open problem.
\end{abstract}

\maketitle

\section{Introduction}
\label{sec:intro}

At the latest with Voiculescu's invention of \emph{freeness} \cite{Voi85}, it became apparent that the ``obvious'' extension of classical stochastic independence, \emph{tensor independence}, is not the only and not always the most suitable concept in inherently noncommutative situations. In fact, \emph{Boolean independence} (not yet under this name) has already featured much earlier in the work of von Waldenfels \cite{vWa73,vWa75}. Those ``noncommutative independences'' share many properties with classical stochastic independence and tensor independence. In particular, under the assumption of independence, mixed moments are uniquely determined and can be calculated from marginal moments (also giving rise to an associated convolution product for probability measures on the real line). Another interesting independence is \emph{monotone independence}, which was discovered by Muraki \cite{Mur01}; this is a non-symmetric independence relation. 

An extremely useful tool when dealing with random variables which have all moments are the corresponding \emph{cumulants}.
The theory of free cumulants, linearizing free additive convolution, was developed by Speicher \cite{Spe94}, see also the book by Nica and Speicher \cite{NiSp06}.\footnote{For a single variable, Voiculescu defined free cumulants and proved their uniqueness already in his seminal paper \cite{Voi85}.} Boolean cumulants were formalized by Speicher and Woroudi \cite{SpWo97}.
Understanding the monotone cumulants took a bit longer, many questions were answered by Hasebe and Saigo \cite{HaSa11}.
The problem in the monotone case is that independence is not in general characterized by vanishing of mixed cumulants.
This is directly related to the non-symmetric nature, as becomes apparent when interpreting moment-cumulant relations via exponential and logarithm maps, as is done in related but different settings by Manzel and Sch{\"u}rmann \cite{MaSc17} (Hopf algebraic) or Ebrahimi-Fard and Patras \cite{EFPa15} (shuffle-algebraic); non-zero mixed cumulants can appear in the Campbell-Baker-Hausdorff series. 

Since the work of Speicher \cite{Spe97}, Ben Ghorbal and Sch{\"u}rmann \cite{bGSc02}, and Muraki \cite{Mur02,Mur03}, we know that the five independence relations for noncommutative random variables, \emph{tensor}, \emph{free}, \emph{Boolean}, \emph{monotone} and \emph{antimonotone} independence, are indeed very special. For these independences, the joint distribution of independent random variables is obtained from the marginal distributions by means of a ``universal product'', i.e.\ a product operation which fulfills a number of natural conditions, including associativity and \emph{universality} (i.e.\ in a specific sense not dependent on the concrete realization of the noncommutative random variables) and a ``factorization for length 2''-condition; and they are the only ones with this property.\footnote{Speicher \cite{Spe97} proved that there are only three \emph{universal calculation rules for mixed moments} in the symmetric case. Ben Ghorbal and Sch{\"u}rmann \cite{bGSc02} axiomatized independences via universal products and showed equivalence to universal calculation rules. Muraki \cite{Mur02,Mur03} extended the results to the non-symmetric setting.} Replacing that ``factorization for length 2''-condition by a positivity condition, a decade later, Muraki \cite{Mur13} proved a similar result with a much simpler proof, while at the same time using a much better motivated assumption, namely that the product operation restricts to a product operation for states on augmented $*$-algebras.\footnote{In the purely algebraic context, i.e.\ without positivity, Muraki's classification was slightly extended by Gerhold and Lachs in \cite{GLa15}, showing that there is a non-symmetric deformation of Boolean independence.} This kind of positivity is also the right condition to study quantum L{\'e}vy processes on dual groups in the sense of Ben Ghorbal and Sch{\"u}rmann \cite{bGSc05}, see also \cite{ScVo14}, where \emph{Schoenberg correspondence} between convolution semigroups of states and conditionally positive generators is proved in this context. In 2014, Voiculescu \cite{Voi14} introduced a new nontrivial extension of free independence, \emph{bifreeness}, for sequences of pairs of random variables, or \emph{pairs of faces} as Voiculescu called the general underlying framework. Taking up on this idea, more examples of \emph{2-faced} or, more generally, \emph{multi-faced} independences have been discovered \cite{Liu19,Liu18p,GuSk19,GHS20,G23}. The general theory of \emph{multi-faced universal products} from which those independences can be obtained was established by Manzel and Sch{\"u}rmann \cite{MaSc17}. It turned out that not all of the examples fulfill the natural positivity condition. Positivity is still enough to assure Schoenberg correspondence in this generalized setting, see \cite{G21p}. In an effort to classify positive multi-faced universal products, two routes have been taken. In \cite{GHU23}, Gerhold, Hasebe, Ulrich completely classified 2-faced universal products which have a natural representation on the tensor product or the free product Hilbert space of the GNS spaces of the factors. In Var\v{s}o's PhD thesis \cite{Var21}, he proved that there are at most 12 two-faced universal products which fulfill additional assumptions of symmetry and a ``combinatorial'' moment cumulant relation (i.e.\ determined by a subset of all two-faced partitions, where more generally weights on two-faced partitions can appear).\footnote{In \cite{Var21}, it was also noticed for the first time the possibility that the moment cumulant relation of a positive universal product might not need to be of combinatorial form, which was indeed confirmed in \cite{GHU23} (cumulants are not discussed explicitly in \cite{GHU23}, but it is apparent that the universal products obtained as deformations can have non-0-1 highest coefficients).} In this article we present, simplify, and extend those results of \cite{Var21}. 

A single-faced independence can trivially be regarded as a two-faced independence, and every two-faced independence is a certain kind of mixture of two single-faced independences. However, neither do those two single-faced independences determine the two-faced independence, nor is it obvious that any combination of single-faced independences can be combined in any way to form a two-faced independence.\footnote{Note that the study of another kind of \emph{mixture} of single-faced independences was initiated by M{\l}otkowski \cite{Mlo04} and received again more attention after work Speicher and Wysozca\'{n}ski \cite{SpWy16} and Ebrahimi-Fard, Patras and Speicher \cite{EPS18} on the corresponding cumulants; this approach is closely related to \emph{graph products} of groups and the corresponding universal products are not associative binary operations.} The main result of this article is to present a family of two-faced symmetric universal products such that every positive symmetric two-faced universal product must belong to that family, we call them \emph{candidates}. This is achieved in three steps. First, we prove necessary conditions for a family of weights on ordered partitions to be the highest coefficients of a positive multi-faced universal product (\Cref{thm:properties-of-highest-coeffs});  second, we determine all permutation invariant weights (= weights on non-ordered partitions) which fulfill those properties (\Cref{cor:classification}), we call such weights here \emph{admissible}; third, we prove that admissible weights are always the highest coefficients of a (uniquely determined) symmetric multi-faced universal product (\Cref{thm:reconstruction}). The family of candidates consists of (identifying an independence with its underlying universal product, and disregarding the difference between a 2-faced independence and its image under swapping the faces)
\begin{itemize}
\item 2-faced continuous 1-parameter deformations of free, tensor and bifree independence (positivity is proved in \cite{GHU23}),
\item a tensor-free independence (positivity is not known),
\item a new free-free and a new tensor-tensor independence, different from the trivial ones, bifreeness, and their deformations (positivity is not known),
\item tensor-Boolean, free-Boolean and Boolean independence; positivity for those is also covered in \cite{GHU23}, for free-Boolean it was first shown by Liu \cite{Liu19} and for Boolean independence positivity is of course well-known.
\end{itemize}
We call the independences which are not realized in \cite{GHU23}, i.e.\ those whose positivity is yet unknown, \emph{exceptional}.

We prove many of the preliminary results for the general symmetric multi-faced case. \Cref{thm:properties-of-highest-coeffs}, where we find necessary conditions on weights to arise as highest coefficients of a universal product, is even formulated for not necessarily symmetric products and could be used as a starting point for a more general classification including multi-faced universal products based on monotone independence, such as for example bimonotone independence (of type II) as defined in \cite{G23,GHS20}.

It easily follows from the main result that there are no non-trivial positive and symmetric trace preserving universal products (\Cref{rem:trace-preserving}) and that tensor independence and bifreeness are the only two positive symmetric 2-faced independences which allow to define a convolution of probability measures on $\mathbb R^2$ (\Cref{rem:convolution-R2}).

Among our additional results, we characterize when a positive symmetric multi-faced universal product is unit preserving (\Cref{prop:unit-preserving}), i.e.\ when it can be defined consistently for arbitrary unital algebras (in the other cases, the product operation is only defined for linear functionals on augmented algebras). This is indeed the case for the three continuous families and the four (three up to swapping the faces) exceptional cases.  Furthermore, we establish a simplified mixed moment formula for the special \emph{combinatorial} case where the highest coefficients are only 0 or 1, so that the moment cumulant relation is simply governed by a specific set of partitions (\Cref{thm:combinatorial-mixed-moments}). 

The outline of the article is as follows. In \Cref{sec:prelims,sec:universal-products,sec:partitions}, we introduce the basic concepts, in particular multi-faced universal products and multi-faced partitions. In \Cref{sec:necessary-conditions} we prove the necessary conditions for a family of weights to be the highest coefficients of a positive multi-faced universal product (symmetric or not). In \Cref{sec:classification} we show that those necessary conditions allow us to obtain a concrete list of candidates for symmetric and positive two-faced universal products. In \Cref{sec:moment-cumulant-relations} we give an introduction to Manzel and Sch{\"u}rmann's cumulant theory, adapted to the relevant special case of symmetric multi-faced independences.
In \Cref{sec:reconstruction} we prove, using cumulants, that in the symmetric case the conditions exhibited in \Cref{sec:necessary-conditions} are sufficient to reconstruct a universal product in the algebraic sense (with a simplified formula in the combinatorial case), but it remains open whether these universal products are automatically positive. 
Finally, we characterize in \Cref{unit-preserving} which universal products in our list are unit preserving.
In \Cref{sec:outlook} we name four tasks which have to be completed in order to achieve a complete classification of positive multi-faced universal products.

A comparison between this article and corresponding results in Var\v{s}o's PhD thesis \cite{Var21} can be found in \Cref{sec:comparison}.

\section{Preliminaries and notation}
\label{sec:prelims}

We will have to deal a lot with tuples of all kinds, so we introduce some useful notation. Let $X$ and $Y$ be arbitrary sets. For any natural number $n$, denote by $[n]$ the set $\{1,\ldots, n\}$. For an $n$-tuple $\mathbf t=\bigl(\mathbf t(1),\ldots, \mathbf t(n)\bigr)\in X^n$ and a subset $I=\{i_1<\ldots < i_k\}\subset[n]$, we define the \emph{restricted tuple} $\mathbf t\restriction I:=\bigl(\mathbf t(i_1),\ldots,\mathbf t(i_k)\bigr)$. Two tuples $\mathbf t\in X^n, \mathbf s\in Y^n$ of the same length may be combined to form the tuple $\mathbf t\times \mathbf s\in (X\times Y)^n$ with $(\mathbf t\times \mathbf s)(i)=\bigl(\mathbf t(i),\mathbf s(i)\bigr)$, and conversely, every tuple in $(X\times Y)^n$ is of that form. The set of $n$-tuples of arbitrary length $n$ is denoted $X^*=\bigcup_{n\in\mathbb N_0}X^{n}$. When a set $X$ does not carry any multiplicative structure, we might use the \emph{word notation}, $\mathbf t(1)\cdots \mathbf t(n):=\bigl(\mathbf t(1),\ldots,\mathbf t(n)\bigr)\in X^{n}$. The entries of a tuple $\mathbf t$ might be written $t_i$ instead of $\mathbf t(i)$ from time to time; or we might use $\mathbf t$ as a shorthand for $(t_1,\ldots, t_n)$ without further comment when the $t_i$ have been around before.

An \emph{algebra} means a complex associative algebra, not necessarily unital. The \emph{free product} of algebras $A_1,A_2$ is denoted $A_1\sqcup A_2$, reminding of the fact that this is the coproduct in the category of algebras: for arbitrary algebra homomorphisms $h_i\colon A_i\to B$, there is a unique algebra homomorphism $h_1\sqcup h_2\colon A_1\sqcup A_2\to B$ with $h_1\sqcup h_2\restriction A_i=h_i$. We use the same symbol $\sqcup$ to denote the canonical homomorphism $h_1\sqcup h_2\colon A_1\sqcup A_2\to B_1\sqcup B_2$ when $h_i\colon A_i\to B_i$, it should always be clear from the context which codomain is meant.

For a vector space $V$, we denote by $T_0(V)=\bigoplus _{n\in\mathbb N}V^{\otimes n}$ the (non-unital) free algebra over $V$. We will identify $T_0(V_1\oplus V_2)=T_0(V_1)\sqcup T_0(V_2)$ without further commenting. Furthermore, $T_0(V_1)\oplus T_0(V_2)$ is identified with the corresponding subspace of $T_0(V_1\oplus V_2)$ and linear functionals $\psi$ on $T_0(V_1)\oplus T_0(V_2)$ are identified with linear functionals $T_0(V_1\oplus V_2)$ by extending them as the 0-functional to the canonical complement, i.e.\
\[\psi(v_1\otimes\cdots\otimes v_k):=
  \begin{cases}
    \psi(v_1\otimes\cdots\otimes v_k) & \text{if $\forall i:v_i\in V_1$  or  $\forall i:v_i\in  V_2$,}\\
    0 & \text{if $\exists i,j:v_i\in  V_1,v_j\in V_2$.}
  \end{cases}
\]
In particular, this convention applies to the direct sum of two linear functionals $\psi_i\colon T_0(V_i)\to\mathbb C$, i.e.\ we identify $\psi_1\oplus\psi_2$ with the linear functional on $T_0(V_1\oplus V_2)$ given by 
\begin{align}\label{eq:direct-sum-extended}
  \psi_1\oplus\psi_2(v_1\otimes\cdots\otimes v_k)=
  \begin{cases}
    \psi_1(v_1\otimes\cdots\otimes v_k) & \text{if $\forall i:v_i\in V_1$,}\\
    \psi_2(v_1\otimes\cdots\otimes v_k) & \text{if $\forall i:v_i\in  V_2$,}\\
    0 & \text{if $\exists i,j:v_i\in  V_1,v_j\in V_2$.}
  \end{cases}
\end{align}
The unital free algebra is denoted $T(V)=\bigoplus _{n\in\mathbb N_0}V^{\otimes n}$, and this unital algebra is the unitization of $T_0(V)$. 

For the rest of this article, if not explicitly mentioned otherwise, $\mathcal F$ denotes a fixed finite set, whose elements we call \emph{faces} or \emph{colors}. We could of course assume $\mathcal F=[m]$ for $m\in\mathbb N$, but since there will be a lot of integers around, we prefer to use more abstract symbols. We mostly use squared symbols such as $\wsquare,\bsquare$ to denote arbitrary elements of $\mathcal F$. If there are exactly two faces, we assume $\mathcal F=\{\wcol,\bcol\}$.    

A multi-faced (or $\mathcal F$-faced\footnote{We will usually write multi-faced instead of $\mathcal F$-faced. Nevertheless, use of the term always refers to the same fixed set of faces $\mathcal F$.}) algebra is an algebra $A$ that is freely generated by given subalgebras $A^\bsquare$, $\bsquare\in\mathcal F$ (the \emph{faces} of $A$), i.e.\ the canonical algebra homomorphism $\bigsqcup_{\bsquare\in\mathcal F} A^\bsquare\to A$ is an isomorphism; this is indicated by writing $A=\bigsqcup_{\bsquare\in\mathcal F} A^\bsquare$.  A \emph{multi-faced algebra homomorphism} is an algebra homomorphism $h\colon A\to B$ between multi-faced algebras $A,B$ with $h(A^\bsquare)\subset B^\bsquare$. We consider the free product of multi-faced algebras again a multi-faced algebra with faces $(A\sqcup B)^\bsquare:=A^\bsquare\sqcup B^\bsquare$.  Note that the free product of multi-faced algebras is the coproduct in the category $\mathrm{Alg}_{\mathcal F}$ of multi-faced algebras with multi-faced algebra homomorphisms, i.e.\ for every pair of multi-faced algebra homomorphisms $h_i\colon A_i\to B$ there is a unique multi-faced algebra homomorphism $h_1\sqcup h_2\colon A_1\sqcup A_2\to B$ restricting to $h_i$ on $A_i$, respectively for $i=1,2$. 

A multi-faced $*$-algebra is a multi-faced algebra with an involution such that each face is a $*$-subalgebra. Of course, the free product of multi-faced $*$-algebras is again a multi-faced $*$-algebra in the obvious way and the free product of multi-faced $*$-homomorphisms is a $*$-homomorphism.

We say that a linear functional $\varphi\colon A\to\mathbb C$ defined on a multi-faced $*$-algebra is a \emph{restricted state} if its unital extension to the unitization of $A$ is a state (or, equivalently, positive).

\section{Universal products}
\label{sec:universal-products}

\begin{definition}[Cf.\ {\cite[Rem.~3.4]{G21p}}]\label{def:universal-product}
  A \emph{multi-faced universal product} is a binary product operation for linear functionals on multi-faced algebras (with an a priori fixed set of faces $\mathcal F$) which associates with functionals $\varphi_1,\varphi_2$ on multi-faced algebras $A_1,A_2$, respectively, a functional $\varphi_1\odot\varphi_2$ on $A_1\sqcup A_2$ such that
  \begin{itemize}
  \item $(\varphi_1\circ h_1)\odot(\varphi_2\circ h_2)=(\varphi_1\odot\varphi_2)\circ (h_1\sqcup h_2)$ for all multi-faced algebra homomorphisms $h_i\colon B_i\to A_i$ (universality)
  \item $(\varphi_1\odot\varphi_2)\odot\varphi_3 = \varphi_1\odot(\varphi_2\odot\varphi_3)$ (associativity)
  \item $(\varphi_1\odot\varphi_2)\restriction A_i=\varphi_i$ (restriction property).
  \end{itemize}
  The product is called
  \begin{itemize}
  \item \emph{symmetric} if $\varphi_1\odot\varphi_2=\varphi_2\odot\varphi_1$,
  \item \emph{positive} if the product of restricted states on multi-faced $*$-algebras is a restricted state on the free product $*$-algebra.
  \end{itemize}
\end{definition}

Note that we made several implicit identifications between isomorphic free products in the last definition. For a more detailed discussion see \cite{G21p}.

Universal products have been invented to encode independences. In the single-faced case, this has been worked out by Ben Ghorbal and Sch{\"u}rmann \cite{bGSc02}. The multi-faced case is covered by \cite{MaSc17} together with the categorical considerations from \cite{Fra06} and \cite{GLS22}. In a nutshell, given a universal product $\odot$ and a linear functional $\Phi$ on an algebra  $\mathcal A$, algebra homomorphisms $j_\kappa\colon B_\kappa\to\mathcal A$, $\kappa\in[k]$, defined on multi-faced algebras $B_\kappa$, are called $\odot$-independent w.r.t\ $\Phi$ if
\[\Phi\circ (j_1\sqcup\cdots \sqcup j_k) = (\Phi\circ j_1) \odot \cdots  \odot (\Phi\circ j_k),\]
or, in other words, if the joint distribution of the noncommutative random variables $j_\kappa$ coincides with the universal product of their marginal distributions. This induces the usual definitions of independence for $\mathcal F$-tuples of elements or of subalgebras of $\mathcal A$. 
In the remainder of this article, we will not work with the independences themselves, but solely with the underlying universal products, so we refrain from giving more details here.

We will make extensive use of the ``Central Structural Theorem'' for universal products \cite[Theorem~4.2]{MaSc17}. Before we present a simplified version of it adapted to the special case of positive multi-faced universal products, we introduce some more notation and give an example.  

Let $A_1,\ldots, A_k$ be multi-faced algebras and $A=A_1\sqcup \cdots \sqcup A_k$ (i.e.\ we identify the $A_i$ with subalgebras of their free product). For $\mathbf s=\mathbf b\times \mathbf f\in ([k]\times \mathcal F)^n$, we denote
\[A^{\mathbf s}:=\left\{a_1\cdots a_n\in A : a_i\in A_{\mathbf b(i)}^{\mathbf f(i)} \right\}.\]
Note that the $A^{\mathbf s}$ are not necessarily pairwise disjoint.\footnote{Indeed, if $\mathbf s(i)=\mathbf s(i+1)$, then $A^{\mathbf s}\subseteq A^{\mathbf s \restriction \{1,\ldots, i-1,i+1,\ldots,n\}}$ because $A_{\mathbf b(i)}^{\mathbf f(i)}$ is a subalgebra of $A$. A typical way to deal with this is to only consider alternating sequences, i.e.\ demand $\mathbf s(i)\neq\mathbf s(i+1)$ for all $i\in [n-1]$. However, it does not cause problems to formulate the subsequent statements for all $A^{\mathbf s}$, so we decided to do so.} Elements of $[k]^n$ are referred to as \emph{block structures} and elements of $\mathcal F ^n$ are called \emph{face structures}.

For $\mathbf s=\mathbf b\times\mathbf f\in ([k]\times \mathcal F )^n$, put $\beta_\kappa(\mathbf s):=\{\ell\in[n]:\mathbf b(\ell)=\kappa\}$. We call a set partition $\pi$ of $[n]$  \emph{adapted to $\mathbf s$}, and write $\pi\prec \mathbf s$, if the following two conditions are met:
\begin{itemize}
\item each block $\beta\in\pi$ is contained in some $\beta_\kappa(\mathbf s)$; in other words, $\pi$ is a refinement of the set partition $\sigma=\{\beta_1(\mathbf s),\ldots ,\beta_k(\mathbf s)\}$ (to adhere strictly to the usual definition of set partition, empty blocks should be removed from $\sigma$)
\item if $\mathbf s(i)=\mathbf s(i+1)$, then $i,i+1$ belong to the same block of $\pi$. 
\end{itemize}
Note that, obviously, $\sigma$ is the maximal partition (w.r.t.\ refinement order) adapted to $\mathbf s$.

Given a multi-faced universal product $\odot$, we define its \emph{linearized part} as
\[\varphi_1\boxdot\cdots\boxdot\varphi_k(a):=\left.\frac{\partial^k}{\partial t_1\cdots \partial t_k}(t_1\varphi_1)\odot\cdots\odot(t_k\varphi_k)(a)\right\vert_{\mathbf{t}=0}\]
(that this expression is well-defined should be understood as part of the following theorem).

\begin{example}\label{exp:deformed-tensor}
  The deformed tensor product $\odot=\uttprod{}{\zeta}{}{\phantom{\zeta}}=\uttprod{\zeta}{\zeta}{1}{1}$ according to \cite[Proposition 5.10(1) and Example 5.7]{GHU23}\footnote{In the notation of \cite{GHU23}, $\wcol$ is face $(1)$ and $\bcol$ is face $(2)$.}, $\zeta\in\mathbb T$, can be calculated for arbitrary 2-faced algebras $A_\kappa$, linear functionals $\varphi_\kappa\colon A_\kappa\to\mathbb C$ ($\kappa\in\{1,2\}$), and elements $a=a^{\wcol}_1a_2^{\wcol}a_1^{\bcol}a_2^{\bcol}\in A^{1212\times \wcol\wcol\bcol\bcol}$ as follows, abbreviating $\langle b\rangle:=\varphi_\kappa(b)$ for $b\in A_\kappa$: 
  \begin{align*}
\MoveEqLeft\varphi_1 \odot\varphi_2(a) \\
    &=\langle a_{1}^{\wcol} \rangle \langle a_{2}^{\wcol}\rangle  \langle a_{1}^{\bcol}\rangle \langle a_{2}^{\bcol}\rangle + \bigl( \langle a_{1}^{\wcol}a_{1}^{\bcol}\rangle-\langle a_{1}^{\wcol}\rangle\langle a_{1}^{\bcol}\rangle\bigr)\langle a_{2}^{\wcol}\rangle \langle a_{2}^{\bcol}\rangle+ \langle a_{1}^{\wcol}\rangle\langle a_{1}^{\bcol}\rangle \bigl(\langle a_{2}^{\wcol}a_{2}^{\bcol}\rangle-\langle a_{2}^{\wcol}\rangle\langle a_{2}^{\bcol}\rangle\bigr) \\
    &\quad + \overline{\zeta} \bigl( \langle a_{1}^{\wcol}a_{1}^{\bcol}\rangle-\langle a_{1}^{\wcol}\rangle\langle a_{1}^{\bcol}\rangle\bigr)\bigl( \langle a_{2}^{\wcol}a_{2}^{\bcol}\rangle-\langle a_{2}^{\wcol}\rangle\langle a_{2}^{\bcol}\rangle\bigr)\\
    &=\overline\zeta\cdot \langle a_{1}^{\wcol}a_{1}^{\bcol}\rangle\langle a_{2}^{\bcol}a_{2}^{\wcol}\rangle
      + (1-\overline\zeta)\cdot \langle a_{1}^{\wcol}a_{1}^{\bcol}\rangle\langle a_{2}^{\bcol}\rangle\langle a_{2}^{\wcol}\rangle
      + (1-\overline\zeta)\cdot \langle a_{1}^{\wcol}\rangle\langle a_{1}^{\bcol}\rangle\langle a_{2}^{\bcol} a_{2}^{\wcol}\rangle
    \\
    &\quad - (1-\overline\zeta)\cdot \langle a_{1}^{\wcol}\rangle\langle a_{1}^{\bcol}\rangle\langle a_{2}^{\bcol}\rangle\langle a_{2}^{\wcol}\rangle
  \end{align*}
  Consequently, the linearized part is given by
  \[\varphi_1\boxdot\varphi_2(a) = \overline\zeta\cdot \varphi_1(a_{1}^{\wcol}a_{1}^{\bcol})\varphi_2( a_{2}^{\bcol}a_{2}^{\wcol}).\]  
\end{example}
  Note how the summands in the full expansion in \Cref{exp:deformed-tensor} correspond to partitions adapted to $\mathbf s$; the product element $a$ is divided into some sort of ``subproducts'' which are then evaluated in the appropriate $\varphi_\kappa$. This general pattern is made precise in the following theorem and allows to describe a universal product in terms of the complex coefficients appearing in each summand, which are independent of the involved linear functionals, algebras and algebra elements.

\begin{theorem}[{Adjusted and simplified from \cite[Th.~4.2, Rem.\ 4.3, 4.4]{MaSc17}}]\label{thm:MS:highest-coefficients}\ \newline  
  Let $\odot$ be a positive multi-faced universal product and $k\in\mathbb N$.  Then there are unique coefficients $\alpha^{\pi}_{\mathbf s}$, $\mathbf s\in([k]\times \mathcal F )^*$, $\pi\prec \mathbf s$, such that , for all
  linear functionals $\varphi_\kappa\colon A_\kappa\to \mathbb C$ on multi-faced algebras $A_\kappa$ $(\kappa\in[k])$ and all $a\in A^{\mathbf s}$,
  \begin{align}\label{eq:all-coefficients}
    \varphi_1\odot\cdots\odot\varphi_k(a)= \sum_{\pi\prec\mathbf s}\alpha_{\mathbf s}^{\pi}\cdot\prod_{\kappa\in[k]}\prod_{\substack{\beta \in\pi\\\beta \subset \beta _\kappa(\mathbf s)}}\varphi_\kappa\left(\mathop{\overrightarrow{\prod}}_{\ell\in \beta }
    a_\ell\right).
  \end{align}
  (The symbol $\mathop{\overrightarrow{\prod}}$ indicates that the product is to be taken in the same order as the factors $a_j$ appear in the product $a=a_1\cdots a_n$.)
  
  Putting $\alpha_{\mathbf s}:=\alpha_{\mathbf s}^{\sigma}$ ($\sigma$ the maximal partition adapted to $\mathbf s$), the linearized part is given by
  \begin{align}\label{eq:highest-coefficients}
    \varphi_1\boxdot\cdots\boxdot\varphi_k(a)= \alpha_{\mathbf s}\cdot\varphi_1\left(\mathop{\overrightarrow{\prod}}_{\mathbf b(\ell)=1}
    a_\ell\right)\cdots \varphi_k\left(\mathop{\overrightarrow{\prod}}_{\mathbf b(\ell)=k} a_\ell\right).
  \end{align}
  The $\alpha_\mathbf s^{\pi}$ are called \emph{coefficients} of $\odot$ and the $\alpha_{\mathbf s}$ are called \emph{highest coefficients} of $\odot$. 
\end{theorem}

\begin{proof}
 First assume that $\mathbf s\in ([k]\times\mathcal F )^n$ is \emph{alternating}, i.e.\ $\mathbf s(i)\neq\mathbf s(i+1)$ for $i=1,\ldots, n-1$.  By \cite[Rem.~4.3]{MaSc17}, the formula given in \cite[Th.~4.2]{MaSc17} can be applied.
  For a positive universal product,  \cite[Rem.~4.4]{MaSc17} implies that there is only one summand for each $\pi\prec \mathbf s$, corresponding to the ``right-ordered coefficient'' (i.e.\ the $a_j$ are multiplied in the same order in which they appear as factors in $a$) associated with $\pi$ and $\mathbf s$, denoted $\alpha_\mathbf s^{\pi}$ in this article.

If $\mathbf s$ is not alternating, then we define $\alpha_{\mathbf s}^{\pi}:=\alpha_{\widetilde{\mathbf s}}^{\widetilde\pi}$, where  $\widetilde{\mathbf s}$ is the alternating tuple obtained from $\mathbf s$ merging repeating entries into one, and $\widetilde\pi$ the set partition adapted to $\widetilde s$ induced by $\pi$ in the obvious way. By universality it follows that \eqref{eq:all-coefficients} extends to all $\mathbf s\in ([k]\times\mathcal F )^*$; indeed, if $\widetilde{\mathbf s}=\widetilde{\mathbf b}\times \widetilde{\mathbf f}$ has length $m$ and
  \[a_{1},\ldots ,a_{r_1}\in A_{\widetilde{\mathbf b}(1)}^{\widetilde{\mathbf f}(1)},\quad a_{r_1+1},\ldots ,a_{r_2}\in A_{\widetilde{\mathbf b}(2)}^{\widetilde{\mathbf f}(2)},\quad\ldots,\quad a_{r_{m-1}+1},\ldots ,a_{r_1}\in A_{\widetilde{\mathbf b}(m)}^{\widetilde{\mathbf f}(m)} \] then for the multi-faced algebras $B_\kappa$ ($\kappa\in[k]$) which are freely generated by $x_i\in B_{\widetilde{\mathbf b}(i)}^{\widetilde{\mathbf f}(i)}$ and multi-faced homomorphisms $h_\kappa \colon B_\kappa\to A_\kappa$ defined by
  \[B_{\widetilde{\mathbf b}(i)}^{\widetilde{\mathbf f}(i)}\ni x_i\mapsto a_{r_{i-1}+1}\cdots a_{r_i}\in A_{\widetilde{\mathbf b}(i)}^{\widetilde{\mathbf f}(i)}\quad \text{for $i=1,\ldots,m$}\quad (r_0:=0),\]
  one finds that, using universality for the first equality,
  \begin{align*}
    \varphi_1\odot\cdots\odot\varphi_k(a)
    &=(\varphi_1\circ h_1) \odot\cdots\odot (\varphi_k\circ h_k) (x_1\cdots x_m)\\
    &=  \sum_{\widetilde\pi\prec\widetilde{\mathbf s}}\alpha_{\widetilde{\mathbf s}}^{\widetilde \pi}\cdot\prod_{\kappa\in[k]}\prod_{\substack{\widetilde \beta \in\widetilde \pi\\\widetilde \beta \subset \beta _\kappa(\widetilde {\mathbf s})}}(\varphi_\kappa\circ h_{\kappa})\left(\mathop{\overrightarrow{\prod}}_{\widetilde \ell\in\widetilde \beta }
    x_{\widetilde \ell}\right)\\
    &= \sum_{\pi\prec\mathbf s}\alpha_{\mathbf s}^{\pi}\cdot\prod_{\kappa\in[k]}\prod_{\substack{\beta \in\pi\\\beta \subset \beta _\kappa(\mathbf s)}}\varphi_\kappa\left(\mathop{\overrightarrow{\prod}}_{\ell\in \beta }
    a_\ell\right).
  \end{align*}
  
  For each $\rho\prec \mathbf s$, one can easily construct  multi-faced algebras and linear functionals $\varphi_\kappa\colon A_\kappa\to\mathbb C$ and an element $a\in A^{\mathbf s}$ in such a way that
  \[\prod_{\kappa\in[k]}\prod_{\substack{\beta \in\pi\\\beta \subset \beta _\kappa(\mathbf s)}}\varphi_\kappa\left(\mathop{\overrightarrow{\prod}}_{\ell\in \beta } a_\ell\right)=\delta_{\pi,\rho},\]
  and, thus, $\alpha^{\rho}_{\mathbf s}=(\varphi_1\odot\cdots\odot\varphi_k)(a)$. This shows uniqueness of the coefficients.

  \Cref{eq:highest-coefficients} follows from \Cref{eq:all-coefficients} because the summand corresponding to the maximal partition $\sigma$ is the only one which is linear in each $\varphi_\kappa$. 
\end{proof}

Obviously, the family of coefficients determines the universal product. In fact, it follows from the cumulant theory developed in \cite{MaSc17} that the highest coefficients alone are already enough to determine the universal product. We will come back to this in \Cref{sec:moment-cumulant-relations}.

To end this section, we show that the highest coefficients can be recovered from the linearized part of a universal product using only linear functionals of a particularly well-behaved kind.

\begin{definition}
  A restricted state $\varphi\colon A\to\mathbb C$ on a multi-faced algebra $A$ is called \emph{trivially multi-faced} if for all  $\wsquare,\bsquare\in\mathcal F$ there exists a $*$-isomorphism $a^{\wsquare}\mapsto a^{\bsquare}\colon  A^{\wsquare}\to A^{\bsquare}$ with $\varphi (ab^{\wsquare}c)=\varphi(a b^{\bsquare} c)$ for all $b^{\wsquare}\in A^{\wsquare}$ and all $a,c$ in the unitization of $A$.
\end{definition}

\begin{lemma}\label{lem:trivially-multi-faced}
  For every $\mathbf s=\mathbf b\times\mathbf f\in ([k]\times\mathcal F )^*$, there are trivially multi-faced restricted states $\varphi_\kappa$ on multi-faced $*$-algebras $A_\kappa$ ($\kappa\in[k]$) and an element $a\in A^{\mathbf s}$ with
  $\varphi_\kappa\bigl(\mathop{\overrightarrow{\prod}}_{\mathbf b(\ell)=\kappa} a_\ell\bigr)=1$ for all $\kappa\in[k]$; in particular, for a positive multi-faced universal product  $\odot$ it follows that $\alpha_{\mathbf s}=\varphi_1\boxdot\cdots\boxdot\varphi_k(a)$.
\end{lemma}

\begin{proof}
  Define $A_\kappa^{\bsquare}:=\mathbb C$ and $A_\kappa:=\bigsqcup_{\bsquare\in\mathcal F}A_\kappa^{\bsquare}$. Then $\varphi_\kappa=\bigsqcup_{\bsquare\in\mathcal F}\mathrm{id}\colon A_\kappa\to \mathbb C$ is a state, in particular a restricted state, and trivially multi-faced. Put $a_{\kappa}^{\bsquare}:= 1$ for all $\kappa\in[k]$ and all $\bsquare\in\mathcal F$.
  Now it is easy to see that $\varphi_\kappa(a_{\kappa}^{\bsquare_1}\cdots a_{i}^{\bsquare_m})=1$ for all $m\in\mathbb N$ and $\bsquare_\mu\in\mathcal F$ ($\mu\in[m]$). With $a:=a_{\mathbf b(1)}^{\mathbf f(1)}\cdots a_{\mathbf b(n)}^{\mathbf f(n)}$ the first claim is obvious and the second claim follows from \Cref{thm:MS:highest-coefficients}.
\end{proof}

\section{Partitions}\label{sec:partitions}

In general, a \emph{multi-faced set} is a set $S$ together with a map ${\mathbf{f}}\colon S\to\mathcal F$, the \emph{face structure} of $S$. The subsets $S^\bsquare:={\mathbf{f}}^{-1}(\{\bsquare\})$ are called the faces of $S$. A multi-faced subset of $S$ is just a subset of the underlying set viewed as a multi-faced set with respect to the restricted face structure.

In this article, we only deal with multi-faced sets whose underlying set $S$ is finite and totally ordered; these properties are implicitly assumed whenever we write about multi-faced sets in the following. 

Any word ${\mathbf{f}}={\mathbf{f}}(1)\cdots {\mathbf{f}}(n)\in \mathcal F^*$ defines face structure on $[n]$, $k\mapsto {\mathbf{f}}(k)$, (which we identify with the word ${\mathbf{f}}$) thus turning $[n]$ into a multi-faced set, denoted by $[n]_{\mathbf f}$. Conversely, we associate with a multi-faced set $S=(\{s_1<\ldots< s_n\},{\mathbf{f}})$ the word $|S|:={\mathbf{f}}(s_1)\cdots {\mathbf{f}}(s_n)\in\mathcal F^*$. We choose this on first sight odd notation because the word ${\mathbf{f}}$ plays the same role as the number of elements of a set plays in the single-faced case in the moment-cumulant formulas we are aiming at.

Let $S$ be a multi-faced set and $\sim$ an equivalence relation such that
\begin{itemize}
\item the equivalence classes are intervals,
\item ${\mathbf{f}}$ is constant on equivalence classes. 
\end{itemize}
Then we understand the quotient $S/{\sim}$ as a multi-faced set with the induced total order and face map.
\begin{example}\label{ex:quotients}
  We briefly discuss the two situations that will appear several times in this article.
  \begin{enumerate}
  \item Let $\mathbf {\mathbf{f}}\in\mathcal F^n$ be a face  and $\sim$ the equivalence relation on $[n]$ that identifies two neighboring points $i,i+1$ in the same face, i.e.\ ${\mathbf{f}}(i)={\mathbf{f}}(i+1)$. In this case we write ${\mathbf{f}}/(i\sim i+1)$ for the quotient $[n]_{\mathbf{f}}/\sim$ and denote its elements $\ell$ instead of $\{\ell\}$ for the trivial equivalence classes of $\ell\in[n]\setminus\{i,i+1\}$ and $\{i,i+1\}$ for the two-element equivalence class of $i$ and $i+1$. 
  \item Let $S$ be a multi-faced set and $\sim$ the equivalence relation whose equivalence classes are the maximal intervals on which $\mathbf f$ is constant. We then call the quotient $S_{\mathrm{red}}:=S/{\sim}$ the \emph{reduction of $S$}. In the reduction, neighboring points will always have different faces, so that no further quotienting is possible.
  \end{enumerate}
\end{example}

A partition of a multi-faced set $S$ is a collection of multi-faced subsets whose underlying sets form a set partition. The set of all partitions of a multi-faced set $S$ is denoted $\mathcal P(S)$. An ordered partition of $S$ is a partition of $S$ together with a total order between the blocks. The set of all ordered partitions is denoted $\mathcal P_{<}(S)$.

For a word ${\mathbf{f}}\in\mathcal F^n$, we put $\mathcal P({\mathbf{f}}):=\mathcal P([n]_{\mathbf{f}})$ and $\mathcal P_<({\mathbf{f}}):=\mathcal P_<([n]_{\mathbf{f}})$. We also denote
\[\mathcal P:=\bigcup_{{\mathbf{f}}\in\mathcal F^*}\mathcal P({\mathbf{f}}),\quad \mathcal P_{<}:=\bigcup_{{\mathbf{f}}\in\mathcal F^*}\mathcal P_{<}({\mathbf{f}}).\] 

\begin{example}
  Let $\mathcal F=\{{\wcol},{\bcol}\}$ and 
  consider ${\mathbf{f}}={\wcol}{\bcol}{\bcol}{\wcol}{\bcol}\in\mathcal F^*$. Then  $\pi= \{\beta_1,\beta_2\}$ with $\beta_1=\{1,3,4\}, \beta_2=\{2,5\}$ is an element of $\mathcal P({\mathbf{f}})$ and we have $|\beta_1|={\wcol}{\bcol}{\wcol}, |\beta_2|={\bcol}{\bcol}$. This can be nicely drawn as an \emph{arc diagram}, $\pi=\partxyoxoxyo$. 
\end{example}

In the following we will not distinguish between a partition and its arc diagram. In this article, we mostly use arc-diagrams to denote partitions in $\mathcal P$, i.e.\ without a block-order; the height of the blocks is then completely arbitrary. For a partition in $\mathcal P_<$, the height of the block corresponds to the order between blocks. If the underlying set $S$ is not of the form $[n]_{\mathbf f}$ (typically because it was obtained as a quotient), we draw the diagram for the corresponding partition of $|S|$.

$\mathcal P({\mathbf{f}})$ is a partially ordered set by the order of reverse refinement. The maximum and minimum of $\mathcal P({\mathbf{f}})$ are denoted $1_{\mathbf{f}}$ and $0_{\mathbf{f}}$, respectively, i.e.\ $1_{\mathbf{f}}$ is the one-block partition and in $0_{\mathbf{f}}$ all blocks are singletons.

There is a canonical bijection between $\mathcal P(S/{\sim})$ and the set of $\pi\in\mathcal P(S)$ such that equivalent points of $S$ lie in the same block of $\pi$. 

For a multi-faced partition $\pi$, consider the equivalence relation $\sim$ defined on the underlying multi-faced set $S$ by
\begin{center}
  $s\sim t \mathrel{:\Longleftrightarrow}$ all $r\in S$ with $s\leq r\leq t$ have the same color and belong to the same block of $\pi$.
\end{center}
In other words, $\sim$ is the equivalence relation
 whose equivalence classes are the maximal intervals $I$ of $S$ which fulfill the following two properties:
\begin{itemize}
\item ${\mathbf{f}}$ is constant on $I$;
\item all elements of $I$ belong to the same block of $\pi$. 
\end{itemize}
We define the \emph{reduction of $\pi$} as the induced multi-faced partition $\pi_{\mathrm{red}}$ on $S/{\sim}$. For example, \[\left(\partxyoyoyxo\right)_{\mathrm{red}}=\partxyoyxo.\]
Then $\pi_{\mathrm{red}}$ will not have neighboring legs that are in the same face and in the same block. For $\pi\in\mathcal P_<$, the block order remains unchanged.

For a multi-faced set $S$, we define its mirror image $\overline S$ as the set with one element $\overline s$ for each $s\in S$ (so that $s\mapsto \overline s$ is a bijection) with the face structure $\overline{\mathbf f}(\overline s):=\mathbf f(s)$ and reversed order, i.e.\ $\overline s\leq \overline t\iff s\geq t$. For $\pi\in\mathcal P(S)$, we put $\overline\pi\in\mathcal P(\overline S)$ as the set partition with a block $\overline \beta=\{\overline s_1,\ldots, \overline s_n\}$ for each block  $\beta=\{s_1,\ldots, s_n\}\in\pi$. For example,
\[\overline{\wcol\bcol\bcol\wcol\bcol}=\bcol\wcol\bcol\bcol\wcol,\quad\overline{\left(\partxyoxozyo\right)}={\partyozxoyox}.\]
If $S=[n]_{\mathbf f}$ for $\mathbf f\in\mathcal F^n$, so that the underlying set is $[n]$, we use the convention that $\overline k:=n-k+1$ (i.e.\ we identify $\overline k$ with its image under under the unique strictly increasing map $\overline{[n]}\to[n]$); this has the effect that $\overline{[n]}$ is identified with $[n]$ and $\overline{[n]_{\mathbf f}}=[n]_{\overline {\mathbf f}}$ for $\overline{\mathbf f}=\overline{\mathbf f}(1) \cdots\overline{\mathbf f}(n)=\mathbf f(n) \cdots \mathbf f(1)$ the mirror image of $\mathbf f$. This is clearly in accordance the diagrammatic representation. If $\pi=\{\beta_1<\ldots <\beta_k\}\in \mathcal P_<$, then $\overline\pi$ is defined as before together with the (non-reversed!) block order $\overline\beta_1<\ldots<\overline\beta_k$.

Finally, we introduce a notation for uniting blocks. Let $\pi=\{\beta_1<\ldots <\beta_k\}\in\mathcal P_<(S)$ with blocks $\beta_i,\beta_{i+1}$ that are nearest neighbors for the order on $\pi$. Then we define $\pi_{\beta_i\smile \beta_{i+1}}:=\{\beta_1<\ldots< \beta_{i-1} < \beta_i\cup \beta_{i+1} < \ldots < \beta_k\}$. Similarly, for $\pi\in\mathcal P({\mathbf{f}})$ and arbitrary blocks $\beta_1,\beta_2\in\pi$, $\pi_{\beta_1\smile \beta_2}:=\pi\setminus\{\beta_1,\beta_2\}\cup \{\beta_1\cup \beta_2\}$. For example,
\[\left(\partxyoxozyo\right)_{\{1,3\}\smile\{2,5\}}=\partxxoxoyxo.\]

Let $\mathbf f_i\in\mathcal F^{m_i}$, $i\in[n]$, be face structures and $\mathbf f$ their \emph{concatenation}, i.e.\ $\mathbf f(m_1+\cdots + m_{i-1}+\ell)=\mathbf f_i(\ell)$ for all $i\in[n]$, $\ell\in[m_i]$. 
Given partitions $\pi_i\in\mathcal P(\mathbf f_i)$, we define their \emph{concatenation} as the partition $\pi\in\mathcal P(\mathbf f)$ which has for every block $\beta\in\pi_i$ with $i\in[n]$ a block $\widetilde \beta:=\{\ell:\ell+\sum_{j=1}^{i-1} m_j\in \beta\}$. Roughly speaking, $\pi$ restricts to $\pi_i$ on the legs corresponding to $\mathbf f_i$. For example, the concatenation of $\pi_1=\partxxoxoyxo$ and $\pi_2=\partyozxoyox$ is $\pi=\partxxoxoyxo\partyozxoyox$. We do not define here the concatenation of ordered partitions.

\section{Highest coefficients: necessary conditions}\label{sec:necessary-conditions}

\begin{definition}
  A family of complex numbers $\alpha=(\alpha_\pi)_{\pi\in\mathcal P_{<}}$ is called (family of) \emph{weights on ordered partitions}, a family $\alpha=(\alpha_\pi)_{\pi\in\mathcal P}$ is called (family of) \emph{weights on partitions}. Weights on (ordered) partitions are called \emph{monic} if $\alpha_\pi=1$ for every one-block partition.
\end{definition}

For a family of numbers
\[\alpha_{\mathbf s}:\mathbf s\in([k]\times \mathcal F)^n, k,n\in\mathbb N\]
(as it is for example obtained from a universal product by \Cref{thm:MS:highest-coefficients}) and $\pi=\{\beta_1<\ldots< \beta_k\}\in\mathcal P_<({\mathbf{f}})$ an ordered multi-faced partition with $k$ blocks, we define $\mathbf s_\pi\in ([k]\times \mathcal F)^n$ via $\mathbf s_\pi(\ell):=(\kappa,\bsquare)$ if $\ell\in \beta_\kappa$ and ${\mathbf{f}}(\ell)=\bsquare$ and put
\[\alpha_\pi:=\alpha_{\mathbf s_\pi}.\]
In this way, we associate with each universal product a family of weights on ordered partitions, and we say that the weights of a universal product are its highest coefficients. Note that such weights are always monic.

We say that weights on ordered partitions $\alpha$ are \emph{invariant under permutation of blocks} if
\[\alpha_{\{\beta_1<\ldots< \beta_k\}}=\alpha_{\{\beta_{1'}<\ldots<\beta_{k'}\}} \text{ for every permutation $\kappa\mapsto\kappa'$ of $[k]$}.\]
In this case, define $\alpha_\pi$ for a non-ordered partition $\pi=\{\beta_1,\ldots, \beta_k\}\in\mathcal P({\mathbf{f}})$ simply as the value $\alpha_{\{\beta_1<\ldots < \beta_k\}}$ for an arbitrary ordered partition with the same blocks as $\pi$. In this way, we can identify weights on partitions and weights on ordered partitions which are invariant under block permutation.

\begin{remark}
  It is easy to check that the weights $\alpha$ coming from a universal product according to \Cref{thm:MS:highest-coefficients} are invariant under permutation of blocks if and only if the universal product is symmetric. 
\end{remark}

The question we wish to answer is the following: under which conditions on the weights $\alpha$ is there a (positive) universal product $\odot$ with highest coefficients $\alpha$? The next theorem yields some necessary conditions. 

\begin{theorem}\label{thm:properties-of-highest-coeffs}
  Let $\odot$ be a positive multi-faced universal product. Then the highest coefficients fulfill:
  \begin{enumerate}[label=\textnormal{(\roman*)}]
  \item\label[condition]{it:highest-coeff:1-block-partitions} $\alpha_{1_{\mathbf{f}}}=1$ for all ${\mathbf{f}}\in\mathcal F^*$.
  \item\label[condition]{it:two-point-partition} $\alpha_{(\{\{1\}<\{2\}\},{\mathbf{f}})}=\alpha_{(\{\{2\}<\{1\}\},{\mathbf{f}})}=1$ for every ${\mathbf{f}}\in \mathcal F^2$.
  \item\label[condition]{it:highest-coeff:red} $\alpha_\pi=\alpha_{\operatorname{red}(\pi)}$.
  \item\label[condition]{it:highest-coeff:split} Suppose $\pi\in\mathcal P_<({\mathbf{f}})$ has blocks $\beta_1<\beta_2$ that are nearest neighbors for the order of $\pi$ and have neighboring legs in the same face, i.e.\ there exist $i\in \beta_1$, $j\in \beta_2$, $|i-j|=1$, ${\mathbf{f}}(i)={\mathbf{f}}(j)$.
    Then
    \[\alpha_\pi=\alpha_{\pi_{\beta_1\smile \beta_2}}\cdot \alpha_{\{\beta_1<\beta_2\}}\]
  \item\label[condition]{it:highest-coeff:change-color} $\alpha_\pi=\alpha_\sigma$ whenever $\pi$ and $\sigma$ only differ in the faces of extremal legs.
  \item\label[condition]{it:highest-coeff:symmetry} $\alpha_{\overline\pi}=\overline{\alpha_{\pi}}$.
  \end{enumerate}
\end{theorem}

\begin{proof}
  Recall the definition of $\mathbf s_\pi\in([k]\times \mathcal F)^*$ for $\pi=(\beta_1<\ldots< \beta_k)\in\mathcal P_<$ from the beginning of this section.  By \Cref{lem:trivially-multi-faced}, we can express each coefficient $\alpha_\pi$ as
  \[\alpha_\pi=\varphi_1\boxdot\cdots\boxdot\varphi_k(a)\]
  with $a\in A^\pi:=A^{\mathbf s_\pi}$ and $(\varphi_1\otimes\cdots\otimes\varphi_k)(a_\pi)=1$, where $a_\pi:=\bigl(\mathop{\overrightarrow{\prod}}_{\ell\in \beta_1} a_\ell\bigr)\otimes\cdots\otimes\bigl(\mathop{\overrightarrow{\prod}}_{\ell\in \beta_k} a_\ell\bigr)$. We will freely use this notation in the rest of the proof.
  
  \ref{it:highest-coeff:1-block-partitions} follows from the restriction property in \Cref{def:universal-product}. \ref{it:highest-coeff:red} holds by definition of the non-reduced coefficients in the proof of  \Cref{thm:MS:highest-coefficients}. For \ref{it:highest-coeff:split} we have to carefully analyse the linearized universal product. If $\pi$ has neighboring blocks $\beta_1<\beta_2$ with neighboring legs in face $\bsquare\in \mathcal F$, then $a\in A^{\pi}$ implies that $a=a_1\cdots a_{r}a_{r+1}\cdots a_n$ with $a_{r}\in A_i^{\bsquare},a_{r+1}\in A_j^{\bsquare}$ with $|i-j|=1$. Without loss of generality, assume $j=i+1$. Then
  \begin{align*}
    \alpha_\pi
    &=\varphi_1\boxdot\cdots\boxdot\varphi_k (a)\\
    &=\left.\frac{\partial^k}{\partial t_1\cdots \partial t_k} \Bigl((t_1\varphi_1)\odot\cdots \odot\bigl((t_i\varphi_i)\odot(t_{i+1}\varphi_{i+1})\bigr)\odot\cdots\odot\ ( t_k\varphi_k )\Bigr)(a)\right\vert_{\mathbf t=0}.
  \end{align*}
  Evaluating the full coefficient formula, \Cref{eq:all-coefficients}, for the universal product of the $k-1$ functionals
  \[\psi_\ell:=
    \begin{cases}
      t_\ell\varphi_\ell& \text{if $\ell<i$,}\\
      t_i\varphi_i\odot t_{i+1}\varphi_{i+1}&\text{if $\ell=i$,}\\
      t_{\ell+1}\varphi_{\ell+1}&\text{if $\ell>i$}
    \end{cases}
  \]
  every summand will contain a factor $F=\psi_i(\cdots a_ra_{r+1}\cdots)$ because the two factors $a_r,a_{r+1}$ are from the same block and face and therefore have to be treated as one. Summands with more factors containing $\psi_i$ vanish in the linearization procedure. Therefore, we obtain
   \begin{align*}
     \MoveEqLeft\left.\frac{\partial^k}{\partial t_1\cdots \partial t_k} \Bigl((t_1\varphi_1)\odot\cdots \odot\bigl((t_i\varphi_i)\odot(t_{i+1}\varphi_{i+1})\bigr)\odot\cdots\odot\ ( t_k\varphi_k )\Bigr)(a)\right\vert_{\mathbf t=0}\\
     &=\alpha_{\pi_{\beta_1\smile \beta_2}}\left.\frac{\partial^2}{\partial t_i \partial t_{i+1}}  \bigl((t_i\varphi_i)\odot(t_{i+1}\varphi_{i+1})\bigr)(\cdots a_r a_{r+1}\cdots )\right\vert_{\mathbf t=0}\\
     &=\alpha_{\pi_{\beta_1\smile \beta_2}}\cdot \alpha_{\{\beta_1<\beta_2\}}
   \end{align*}
   as claimed

   So far, we have not made significant use of positivity (except that we assumed that wrong ordered coefficients vanish), but positivity is important to prove the remaining two properties.

   \ref{it:highest-coeff:symmetry} follows easily from the fact that positive functionals are hermitian and $a\in A^\pi$ if and only if $a^*\in A^{\overline \pi}$. All we have to do is choose some restricted states $\varphi_1,\ldots,\varphi_k$ and $a\in A^\pi$ with $\varphi_1\otimes\cdots\otimes\varphi_k(a_\pi)=1$, then $\varphi_1\otimes\cdots\otimes\varphi_k((a^*)_{\overline{\pi}})=1$ and we conclude  $\overline{\alpha_\pi}=\varphi_1\boxdot\cdots\boxdot\varphi_k(a^*)=\alpha_{\overline \pi}$.

   To show \ref{it:highest-coeff:change-color}, 
   assume that $\alpha_\pi=\varphi_1\boxdot\cdots\boxdot\varphi_k(a)$, $a=a_1^{\wsquare}a_2\cdots a_n$ with $a_1^{\wsquare}\in A^{\wsquare}_i$ and trivially multi-faced restricted states $\varphi_\kappa$, this is always possible by \Cref{lem:trivially-multi-faced}.
   Then $t_\kappa\varphi_\kappa$ is a restricted state for all $t_\kappa\leq1$, and
   \begin{multline*}
     \left|t_1\varphi_1\odot\cdots\odot t_k\varphi_k\bigl((a_1^{\wsquare}-a_1^{\bsquare})a_2\cdots a_n\bigr)\right|^2\\
     \leq t_i\varphi_i\bigl( (a_1^{\wsquare}-a_1^{\bsquare})^*(a_1^{\wsquare}-a_1^{\bsquare})\bigr)  (t_1\varphi_1\odot\cdots\odot t_k\varphi_k)\bigl((a_2\cdots a_n)^*(a_2\cdots a_n)\bigr)=0
   \end{multline*}
   where $a_1^{\bsquare}$ is the image of $a_1^{\wsquare}$ under the isomorphism $A_i^{\wsquare}\cong A_{i}^{\bsquare}$ making $\varphi_i$ trivially multi-faced. From this the statement for the first leg readily follows. For the corresponding statement for the last leg, we can either apply \ref{it:highest-coeff:symmetry} or perform an analogous computation. 

   Finally, let $\pi=(\{\{1\}<\{2\}\},{\mathbf{f}})$ or $\pi=(\{\{2\}<\{1\}\},{\mathbf{f}})$. By \ref{it:highest-coeff:change-color}, we can assume without loss of generality that ${\mathbf{f}}(1)={\mathbf{f}}(2)$. Therefore, \ref{it:two-point-partition} follows from the single-faced case, which is settled in \cite[Theorem 2.5]{bGSc05}\footnote{In the statement, Ben\,Ghorbal and Sch{\"u}rmann assume ``nondegenerateness'', but the proof does not use this assumption.}.
\end{proof}

\begin{remark}
  Note that the multi-faced universal products of bi-Boolean independence (defined by Gu and Skoufranis \cite{GuSk19}) and bi-monotone independence of type I (defined by Gu, Hasebe and Skoufranis \cite{GHS20}) are not positive. Their associated highest coefficients do not fulfill \ref{it:highest-coeff:symmetry}.
\end{remark}

For the rest of this article, we restrict ourselves to the symmetric case. As noted before, symmetry of the universal product is equivalent to invariance under block-permutation of its highest coefficients, and in this case we denote its highest coefficients $\alpha_\pi$ with $\pi\in\mathcal P$.

\begin{definition}\label{def:admissible}
  A family $\alpha=(\alpha_\pi)_{\pi\in\mathcal P}$ of complex numbers is called \emph{admissible weights} if the corresponding block-permutation invariant family $\alpha_{(\pi,<)}:=\alpha_\pi$ fulfills \textrm{(i)} -- \textrm{(vi)} in \Cref{thm:properties-of-highest-coeffs}; in particular, it fulfills
  \begin{enumerate}
  \item[(iv')]  Suppose $\pi\in\mathcal P({\mathbf{f}})$ has blocks $\beta_1\neq \beta_2$ with neighboring legs $i\in \beta_1$, $i+1\in \beta_2$ of the same face, ${\mathbf{f}}(i)={\mathbf{f}}(i+1)$.
    Then
    \[\alpha_\pi=\alpha_{\pi_{\beta_1\smile \beta_2}}\cdot \alpha_{\{\beta_1,\beta_2\}}.\]  
  \end{enumerate}
\end{definition}
\begin{definition}\label{def:admissible-partitions}
  A set of multi-faced partitions $\Pi\subset \mathcal P$ is called an \emph{admissible set of partitions} if
  \[\alpha_\pi:=
    \begin{cases}
      1&\text{if $\pi\in \Pi$,}\\
      0&\text{otherwise}
    \end{cases}
  \]
  defines admissible weights.\footnote{Note that this is closely related to the definition of a universal class of partitions in \cite{Var21}, but not completely equivalent; the difference is that an admissible set must always contain the partitions $\sqpartxyo$ for all ${\wsquare\bsquare}\in\mathcal F^2$.}
\end{definition}

\begin{notation}
  If $\pi$ is described by a certain arc-diagram $\mathrm{Diag}$, we will write $\alpha(\mathrm{Diag})$ instead of $\alpha_\pi$. Also, we will use a grey square (or circle in the 2-faced case) to indicate that the color of the extremal legs is arbitrary. For example $\alpha\left(\sqpartxyzyozox\right)=\alpha_\pi$ for $\pi=(\{\{1,6\},\{2,4\},\{3,5\}\},{\wsquare}{\wsquare}{\wsquare}{\bsquare}{\bsquare}{\bsquare})$ or any other $\pi$ with the same set partition and the same coloring of the non-extremal legs 2,3,4,5.
\end{notation}

\begin{observation}\label{obs:admissible-weights-and-sets}
  Let $(\alpha_\pi)_{\pi \in \mathcal{P}}$ be admissible weights. Then $\Pi_\alpha=\{\pi:\alpha_\pi\neq0\}$ is an admissible set of partitions. 
  There are, however, admissible families with $\alpha_\pi\notin\{0,1\}$ for some $\pi\in\mathcal P$. Indeed, \Cref{exp:deformed-tensor} in particular shows that, for $\alpha$ the highest coefficients of the deformed tensor product $\uttprod{}{\zeta}{}{\phantom{\zeta}}$ with $\zeta\neq 1$, one finds $\alpha(\crosscircbullet)=\overline\zeta\notin\{0,1\}$.
\end{observation}

\begin{observation}\label{obs:operations-for-partitions}
  A set $\Pi\subset\mathcal P$ of partitions is admissible if and only if $\Pi$ contains the partitions 
  \begin{enumerate}[itemindent=.5em]
  \item[(P-i)]\label{it:P-i} $1_{\mathbf{f}}$ for all ${\mathbf{f}}\in\mathcal F^{*}$
  \item[(P-ii)]\label{it:P-ii} $\sqpartxyo$ for all ${\wsquare\bsquare}\in\mathcal F^2$
  \end{enumerate}
  and is closed under the following operations used in \cite{Var21}:
  \begin{enumerate}[itemindent=.5em]
  \item[(P-iii)] \emph{double} a leg, including its color 
  \item[(P-iii)'] \emph{merge} two neighboring legs of the same color in the same block into one
  \item[(P-iv)] \emph{unite} two blocks which have neighboring legs of the same color into one block, $\pi\mapsto \pi_{\beta_1\smile \beta_2}$
  \item[(P-iv)'] \emph{remember} a two-block partition formed by two blocks with neighboring legs of the same color, $\pi\mapsto \{\beta_1,\beta_2\}$
  \item[(P-iv)''] \emph{replace} a block of a partition from $\Pi$ by a two-block partition from $\Pi$ (of the same underlying multi-faced set as the original block) such that the blocks have neighboring legs of the same color, $(\pi_{\beta_1\smile \beta_2},\{\beta_1,\beta_2\})\mapsto \pi$
  \item[(P-v)] \emph{mirror} a partition, $\pi\mapsto \overline\pi$
  \item[(P-vi)] \emph{change color} of an extremal leg of a partition from $\Pi$
  \end{enumerate}
\end{observation}

Given any partitions $\pi_1,\ldots, \pi_n\in \mathcal P$, we denote by $\langle \pi_1,\ldots,\pi_n\rangle$ the minimal admissible set of partitions that contains all $\pi_i$.  We say that $\langle \pi_1,\ldots,\pi_n\rangle$ is \emph{generated} by $\pi_1,\ldots, \pi_n$; note that $\langle \pi_1,\ldots,\pi_n\rangle$ indeed consists of those partitions in $\mathcal P$ which can be obtained in finitely many steps by applying the operations of \Cref{obs:operations-for-partitions} to the partitions $1_{\mathbf{f}}$ ($\mathbf f\in\mathcal F^*$), $\sqpartxyo$ ($\wsquare\bsquare\in\mathcal F^2$), and $\pi_1,\ldots, \pi_n$.

\section{Partial classification of symmetric positive independences}\label{sec:classification}

In this section we determine all admissible families $(\alpha_\pi)_{\pi\in\mathcal P}$.

\begin{definition}
  Let $\pi$ be a partition.
  \begin{itemize}
  \item A leg $\ell$ is called \emph{inner} if there exist legs $i<\ell<j$ and a block $\beta\in\pi$ with $i,j\in \beta$ and $\ell\notin \beta$. Otherwise it is called \emph{outer}.  
  \item Two legs $\ell,\ell'$ are called \emph{connected} if they lie in the same block or if there is a sequence of blocks $\ell\in \beta_1,\ldots, \beta_n\ni \ell'$ such that there is a crossing between $\beta_k$ and $\beta_{k+1}$. Roughly speaking, $\ell$ and $\ell'$ are connected if and only if one can move from $\ell$ to $\ell'$ going only along the lines of the diagram associated with $\pi$.
  \end{itemize}
\end{definition}

We start by describing some simple consequences of the defining properties of admissible families of coefficients.

\begin{lemma}\label{lem:admissible-weights-simple-properties}
  Let $(\alpha_\pi)_{\pi\in\mathcal P}$ be admissible weights.
  \begin{enumerate}
  \item\label{it:interval-coefficients} $\alpha_\pi=1$ for all interval partitions $\pi$.
  \item\label{it:concat} Let $\pi$ be the concatenation of $\pi_1,\ldots,\pi_n$. Then $\alpha_\pi=\prod \alpha_{\pi_i}$.
  \item\label{it:splitting} $\alpha_\pi=\alpha_\sigma$ when $\sigma$ is obtained by replacing one leg $\ell$ by two copies and splitting the block $\beta\ni \ell$ into $\beta_1$ and $\beta_2$, where $\beta_1$ contains the first copy and all legs of $\beta$ smaller than $\ell$ and $\beta_2$ contains the second copy and all legs of $\beta$ larger than $\ell$. We say that $\sigma$ is obtained by \emph{splitting} $\beta$ at $\ell$.
  \item \label{it:collapse-outer-legs}
    $\alpha_\pi=\alpha_\sigma$ when $\sigma$ is obtained replacing an arbitrary number of connected outer legs by a single outer leg of arbitrary color. We call this process \emph{collapsing} the outer legs. 
  \end{enumerate}
\end{lemma}
\begin{proof}\ 
  \begin{enumerate}[beginpenalty=10000]
  \item This is easily proved by induction. For a two-block interval partition $\pi$, we can consecutively change color of the extremal legs and merge them with their neighboring legs until we reach $\alpha_\pi=\alpha\left(\sqpartxoyo\right)=1$. It is worth noting that for this step we needed to change the color of both extremal legs.

    Assume that the statements holds for $(n-1)$-block interval partitions and let $\pi=(\{\beta_1,\ldots, \beta_n\},{\mathbf{f}})$ be an interval partition with $n>2$ blocks. Starting similar as before, we can without loss of generality assume that $\beta_1=\{1\}$ and ${\mathbf{f}}(1)={\mathbf{f}}(2)$. In that case, we find that $\alpha_\pi=\alpha_{\pi_{\beta_1\smile \beta_2}}\cdot\alpha_{\{\beta_1, \beta_2\}}=1$.  
  \item Clearly, it is enough to prove the claim for $n=2$. We prove the claim by induction on the number of blocks $|\pi|$. If $|\pi|=2$, then $|\pi_1|=|\pi_2|=1$ and the three partitions are interval partitions, in particular $\alpha_{\pi}=1=\alpha_{\pi_1}\alpha_{\pi_2}$. If $|\pi|>2$, then $|\pi_1|>1$ or $|\pi_2|>1$. In case $|\pi_1|>2$, let  $1\in \beta_1\in \pi_1$. We can assume without loss of generality that $2\in \beta_2$ belongs to a different block $\beta_1\neq \beta_2\in\pi_1$ and ${\mathbf{f}}(1)={\mathbf{f}}(2)$; if those conditions are not met, it does not change the coefficient to change the color of the first leg to match the color of the second leg and merge them into one until we are in the described situation. Now we find $\alpha_{\pi_1}=\alpha_{{\pi_1}_{\beta_1\smile \beta_2}}\cdot \alpha_{\{\beta_1,\beta_2\}}$ and $\alpha_{\pi}=\alpha_{{\pi}_{\beta_1\smile \beta_2}}\cdot \alpha_{\{\beta_1,\beta_2\}}$. Of course, $|{\pi}_{\beta_1\smile \beta_2}|=|\pi|-1$, so we may assume that the statement holds for ${\pi}_{\beta_1\smile \beta_2}$ which is the concatenation of ${\pi_1}_{\beta_1\smile \beta_2}$ and $\pi_2$. Altogether, \[\alpha_{\pi}=\alpha_{{\pi}_{\beta_1\smile \beta_2}}\cdot \alpha_{\{\beta_1,\beta_2\}}=\alpha_{{\pi_1}_{\beta_1\smile \beta_2}}\cdot\alpha_{\pi_2}\cdot\alpha_{\{\beta_1,\beta_2\}}=\alpha_{\pi_1}\alpha_{\pi_2}.\]
    If $|\pi_2|>1$, we argue analogously, but we have to change the color of the last leg. 
\item We have $\alpha_\pi=\alpha_\sigma\alpha_{\{\beta_1,\beta_2\}}$, and $\alpha_{\{\beta_1,\beta_2\}}=1$ by (1).
\item Decompose $\pi$ into a concatenation of \emph{irreducible} $\pi_1,\ldots,\pi_n$, i.e.\ no $\pi_i$ can be deconcatenated any further. By \Cref{it:concat}, $\alpha_\pi= \prod \alpha_{\pi_i}$. Note that every outer leg of $\pi$ is the outer leg of some $\pi_i$ and that connected outer legs are necessarily in the same block.  For each $\pi_i$, the outer legs can be collapsed by iteratively changing the face of the first or last leg to match the face of its successor or predecessor, respectively, and merging the legs using the fact that the weights don't change when we reduce the partition (\Cref{it:highest-coeff:red} in \Cref{thm:properties-of-highest-coeffs}). After collapsing the outer legs that way, the faces of the outer legs can be changed once more in such a way that the concatenation of the obtained partitions $\sigma_i$ is $\sigma$. It follows, using again \Cref{it:concat}, that $\alpha_\sigma = \prod \alpha_{\sigma_i}=\prod \alpha_{\pi_i}=\alpha_\pi$. 
\end{enumerate}

\end{proof}

It is worth noting that, in the proof of \Cref{it:concat}, we need invariance of the coefficients under changing the faces of both extremal legs. For example, the weights associated with bi-Boolean independence defined in \cite{GuSk19} do not share this property. 

\begin{lemma}\label{lem:coincide-on-2block}
  Two admissible families coincide if and only if they coincide on 2-block partitions. 
\end{lemma}
\begin{proof}
  Assume that $(\alpha_\pi),(\beta_\pi)$ are admissible families with $\alpha_\sigma=\beta_\sigma$ for all 2-block partitions $\sigma$. By definition, the value on 1-block partitions is 1. Given an $n$-block partition $\pi$ with $n>2$, we alternatingly
  \begin{itemize}
  \item change the color of the first leg to match the color of the second leg, cf.\ \ref{it:highest-coeff:change-color},
  \item combine the first two legs into one if they belong to the same block, cf.\ \ref{it:highest-coeff:red},
  \end{itemize}
  to obtain a partition $\tilde\pi$ such that the first two legs of $\pi$ have the same color but belong to different blocks $\beta_1,\beta_2$. Then  $\alpha_\pi=\alpha_{\tilde\pi}$,  $\beta_\pi=\beta_{\tilde\pi}$ by definition of admissible weights. Using \ref{it:highest-coeff:split}, we then have $\alpha_\pi=\alpha_{\pi_{\beta_1\smile \beta_2}}\cdot \alpha_{\{\beta_1,\beta_2\}}$ and $\beta_\pi=\beta_{\pi_{\beta_1\smile \beta_2}}\cdot \beta_{\{\beta_1,\beta_2\}}$, where $\pi_{\beta_1\smile \beta_2}$ is an $(n-1)$-block partition and ${\{\beta_1,\beta_2\}}$ is a 2-block partition. We can iterate the procedure until we obtain $\alpha_\pi,\beta_\pi$ as products of coefficients of the same sequence of 2-block partitions, thus proving the claim. 
\end{proof}

\begin{corollary}\label{lem:coincide-on-2block4leg}
  Two admissible families coincide if and only if they coincide on 2-block partitions of at most four legs.
\end{corollary}

\begin{proof}
  Suppose that $\pi=\{\beta,\gamma\}$ has more than $4$ legs and that the third leg lies in $\beta$. Without loss of generality, we assume that the first leg and the second leg belong different blocks but the same face; if they would belong to the same block, they could be collapsed and the face of the first leg can simply be adapted to that of the second leg without changing the coefficient. Without loss of generality assume that $1\in \beta$. If all legs after the third leg belong to $\gamma$, they are necessarily outer and can be collapsed to reach a partition with four legs. If there is at least one leg from $\beta$ after the third leg, then splitting $\beta$ at the third leg yields a partition $\tilde\pi=\{\beta_{1},\beta_{2},\gamma\}$ where $\beta_{1}=\{1,3\}$ has two legs and $\beta_{2}=\{3'\}\cup (\beta\setminus\{1,3\})$ has exactly one leg less than $\beta$; here $3'$ is the copy of $3$ obtained from splitting such that $3<3'$. Now, $\alpha_\pi=\alpha_{\tilde\pi}=\alpha_{\tilde\pi_{\beta_1\smile \gamma}}\alpha_{\{\beta_1,\gamma\}}$. Obviously, $\tau:=\{\beta_1,\gamma\}$ has strictly less legs than $\pi$. Since the first three legs of $\tilde\pi_{\beta_1\smile \gamma}$ belong to the same block, after collapsing those three legs, we get a partition $\sigma$ with $\alpha_\sigma=\alpha_{\tilde\pi_{\beta_1\smile \gamma}}$ which has one leg less than $\pi$ (one leg more from the splitting are overcompensated by two legs less from collapsing). All in all, $\alpha_\pi=\alpha_\tau\alpha_\sigma$, where both, $\tau$ and $\sigma$ are two-block partitions with a strictly smaller number of legs than $\pi$. This procedure can be iterated until $\alpha_\pi$ is expressed as a product of only 2-block partitions with at most 4 legs.
\end{proof}

\begin{definition}
  We introduce shorthand notations for the \emph{basic coefficients}, where ${\wsquare},{\bsquare}\in\mathcal F$: 
   \begin{align*}
    \nu_{\wsquare}&:=\alpha\left(\sqnestcirc\right), &\xi_{\wsquare}&:=\alpha\left(\sqcrosscirccirc\right),& \nu_{\wsquare\bsquare}&:=\alpha\left(\sqnestcircbullet\right), &\xi_{\wsquare\bsquare}&:=\alpha\left(\sqcrosscircbullet\right)
   \end{align*}
   (Note that $\xi_{\wsquare\wsquare} =\xi_{\wsquare}$, obviously, and $\nu_{\wsquare\wsquare}=\nu_{\wsquare}$, because we can merge neighboring legs of the same face.)
\end{definition}

\begin{corollary}
  Two admissible families coincide if and only if they have the same basic coefficients.
\end{corollary}

\begin{proof}
  A two-block partition with at most four legs is either an interval partition (in which case its coefficient is 1) or it can be reduced by changing color and combining legs to one of the partitions that define the basic coefficients.
\end{proof}

\begin{lemma}\label{lem:relations-for-basic-coefficients}
  We have the following relations between the basic coefficients for all ${\wsquare},{\bsquare}\in\mathcal F$:
  \begin{enumerate}
  \item $\nu_{\wsquare}^2=\nu_{\wsquare},\xi_{\wsquare}^2=\xi_{\wsquare}$
    , i.e.\ $\nu_{\wsquare},\xi_{\wsquare}\in\{0,1\}$,
  \item $t\nu_\wsquare=t$ for $t\in\{\nu_{\wsquare\bsquare},\xi_{\wsquare},\xi_{\wsquare\bsquare}\}$, i.e.\ $\nu_{\wsquare}=0\implies \nu_{\wsquare\bsquare}=\xi_{\wsquare}=\xi_{\wsquare\bsquare}=0$,
  \item $|t|^2t=t$ for $t\in\{\nu_{\wsquare\bsquare},\xi_{\wsquare\bsquare}\}$, i.e.\ $\nu_{\wsquare\bsquare},\xi_{\wsquare\bsquare}\in\{0\}\cup\mathbb T$,
  \item $\nu_{\wsquare\bsquare}\xi_{\wsquare}=\xi_{\wsquare\bsquare}\xi_\wsquare
$, i.e.\ $\xi_{\wsquare}=1\implies \nu_{\wsquare\bsquare}=\xi_{\wsquare\bsquare}$,
  \item $\nu_{\wsquare\bsquare}\xi_{\wsquare\bsquare}=\nu_{\wsquare\bsquare}\xi_{\wsquare\bsquare}\xi_{\wsquare} $, i.e.\ $\xi_{\wsquare}=0\implies \nu_{\wsquare\bsquare}=0$ or $\xi_{\wsquare\bsquare}=0$.
  \end{enumerate}
\end{lemma}

\begin{proof}\ 
  \begin{enumerate}[beginpenalty=10000]
  \item This follows as in the single-faced case, see \cite{Spe97}. Alternatively, this follows easily as a special case $\wsquare=\bsquare$ from the items below.
  \item Consider $\sqnestcircbullet$. Split the inner $\wsquare$-leg and merge it's copy with the outer block to obtain $\alpha\left(\sqnestcircbullet\right)=\alpha\left(\sqnestcircbullet\right)\alpha\left(\sqnestcirc\right)$. The other cases work analogously. 
  \item First note that $\alpha\left(\sqpartxyyoyx\right)=\alpha\left(\sqpartxyyozozx\right)=|\nu_{\wsquare\bsquare}|^2$. This leads to
    \begin{align*}
      |\nu_{\wsquare\bsquare}|^2\nu_{\wsquare\bsquare}=\alpha\left(\sqpartxyyoyx\right)\nu_{\wsquare\bsquare}=\alpha\left(\sqpartxyzzoyx\right)=\nu_{\wsquare\bsquare}\nu_\wsquare=\nu_{\wsquare\bsquare}.
    \end{align*}
    Similarly, $\alpha\left(\sqpartxyxoyx\right)=\alpha\left(\sqpartxyxozoyz\right)=|\xi_{\wsquare\bsquare}|^2$ and hence
    \begin{align*}
      |\xi_{\wsquare\bsquare}|^2\xi_{\wsquare\bsquare}
      = \alpha\left(\sqpartxyxoyx\right) \xi_{\wsquare\bsquare}
      = \alpha\left(\sqpartxyzxozy\right)
      = \xi_{\wsquare\bsquare}\nu_\wsquare=\xi_{\wsquare\bsquare}.
    \end{align*}
  \item This follows from
    \begin{align*}
      \nu_{\wsquare\bsquare}\xi_{\wsquare}=\alpha\left(\sqpartxyzyzox\right)=\alpha\left(\sqpartxyxyox\right)\nu_{\wsquare\wsquare}=\alpha\left(\sqpartxyxzyoz\right)\nu_{\wsquare}=\xi_{\wsquare\bsquare}\xi_\wsquare\nu_\wsquare=\xi_{\wsquare\bsquare}\xi_\wsquare.
    \end{align*}
  \item Reusing parts of the calculation above, we find
    \begin{equation*}
      \nu_{\wsquare\bsquare}\xi_{\wsquare\bsquare}=\alpha\left(\sqpartxyzyozox\right)=\alpha\left(\sqpartxyxyox\right)\nu_{\wsquare\bsquare}=\nu_{\wsquare\bsquare}\xi_{\wsquare\bsquare}\xi_\wsquare. \qedhere 
    \end{equation*}
  \end{enumerate}
\end{proof}

\begin{corollary}\label{cor:basic-partitions}
  Two admissible sets of partitions coincide if and only if they 
  have the same intersection with $\bigl\{\sqnestcirc,\sqnestbullet,\sqnestcircbullet,\sqcrosscirccirc,\sqcrossbulletbullet,\sqcrosscircbullet:\wsquare,\bsquare\in\mathcal F\bigr\}$.
  
  Furthermore, for an admissible set $\Pi$ we have the following implications:
    \begin{enumerate}
  \item  If $\Pi$ contains at least one of the partitions $\sqnestcircbullet$, $\sqcrosscirccirc$, $\sqcrosscircbullet$, then it contains $\sqnestcirc$. 
  \item If $\Pi$ contains at least one of the partitions $\sqnestcircbullet$, $\sqcrossbulletbullet$, $\sqcrosscircbullet$, then it contains $\sqnestbullet$. 
  \item If $\Pi$ contains two of the basic partitions $\sqnestcircbullet,\sqcrosscirccirc,\sqcrosscircbullet$, then $\Pi$ contains all partitions with faces from $\{\wsquare,\bsquare\}$.
  \item If $\Pi$ contains two of the basic partitions $\sqnestcircbullet,\sqcrossbulletbullet,\sqcrosscircbullet$, then $\Pi$ contains all partitions with faces from $\{\wsquare,\bsquare\}$.
  \end{enumerate}
\end{corollary}

\begin{definition}\label{def:admissible-sets-concrete}
  A 2-faced partition $\pi$ is called
  \begin{itemize}
  \item \emph{interval partition} if
    all legs are outer or, equivalently, if all its blocks are intervals;
    $\mathrm{I_{\wcol\bcol}}$ denotes the set of all interval partitions,
  \item \emph{noncrossing} if for all $i<j<k<\ell$ and blocks $\beta, \gamma\in\pi$,
    \[i,k\in \beta,\ j,\ell\in \gamma\implies \beta=\gamma;\]
    $\mathrm {NC_{\wcol\bcol}}$  denotes the set of all noncrossing partitions,
  \item \emph{binoncrossing} if for all $i<j<k<\ell$ and blocks $\beta\neq \gamma\in\pi$,
    \begin{align*}
      i,k\in \beta,\ j,\ell\in \gamma\implies {\mathbf{f}}(j)\neq{\mathbf{f}}(k),\\
      i,\ell\in \beta,\ j,k\in \gamma \implies {\mathbf{f}}(j)={\mathbf{f}}(k);
    \end{align*}
  \item \emph{interval-noncrossing} if it is noncrossing and all $\wcol$-legs are outer;
    $\mathrm{I_\wcol NC_\bcol}$ denotes the set of all interval-noncrossing partitions,
  \item \emph{noncrossing-interval} if it is interval-noncrossing after swapping the colors $\wcol$ and $\bcol$; 
    $\mathrm{NC_\wcol I_\bcol}$ denotes the set of all noncrossing-interval partitions,  
  \item \emph{interval-arbitrary} if
    all $\wcol$-legs are outer; $\mathrm{I_\wcol A_\bcol}$ denotes the set of all interval-arbitrary partitions,
  \item  \emph{arbitrary-interval} if it is interval-arbitrary after swapping the colors $\wcol$ and $\bcol$; 
    $\mathrm{A_\wcol I_\bcol}$ denotes the set of all arbitrary-interval partitions,
  \item \emph{noncrossing-arbitrary} if
    every block that contains an inner $\wcol$-leg is monochrome and does not cross any other block, i.e.\  for all legs $i,j,k,\ell$ and all blocks $\beta\neq\gamma\in\pi$, 
    \[j,k\in\beta,\ i,\ell\in\gamma,\ i<k<\ell,\ \mathbf f(k)=\wcol\implies i<j<\ell,\ \mathbf f(j)=\wcol;\] 
$\mathrm{NC_\wcol A_\bcol}$ denotes the set of all noncrossing-arbitrary partitions,
  \item  \emph{arbitrary-noncrossing} if it is noncrossing-arbitrary after swapping the colors $\wcol$ and $\bcol$; 
    $\mathrm{A_\wcol NC_\bcol}$ denotes the set of all arbitrary-interval partitions,  
  \item \emph{pure noncrossing} 
    if it is noncrossing and all inner blocks are monochrome;
    $\mathrm{pNC_{\wcol\bcol}}$ denotes the set of all pure noncrossing partitions,
  \item \emph{pure crossing} if
    connected inner legs have the same color;
    $\mathrm{pC_{\wcol\bcol}}$ denotes the set of all pure noncrossing partitions,
  \item \emph{arbitrary} without any conditions; the set of all bipartitions is also denoted $\mathrm{A_{\wcol\bcol}}$.
  \end{itemize}
\end{definition}

\begin{theorem}\label{thm:classification-two-faced-admissible-sets} There are exactly 12 admissible sets of 2-faced partitions (9 if we identify a set with the one obtained by simply swapping the two colors), namely those given in \Cref{def:admissible-sets-concrete}. \Cref{fig:hasseset} displays their respective containment by means of a Hasse diagram and gives minimal generating sets of 2-block partitions.
\end{theorem}

\begin{figure}
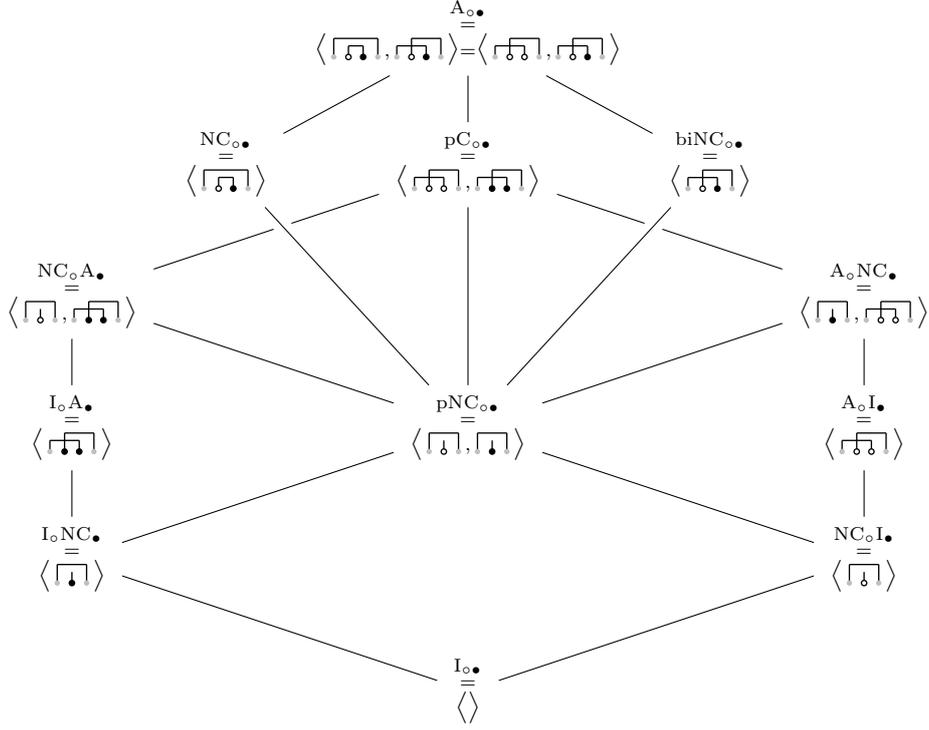

    \centering 
    \usebox{\hasseset} 
    \caption{Hasse diagram of all two-colored admissible sets of partitions.} 
    \label{fig:hasseset}
\end{figure}  

\begin{proof}
  We know that a set obtained from a positive symmetric 2-faced universal product is automatically admissible. 
  Of course, swapping the two colors turns an admissible set into an admissible set. This helps to settle admissibility of a large number of sets in the diagram:
  \begin{itemize}
  \item The sets $\mathrm{I_{\wcol\bcol}}, \mathrm{NC_{\wcol\bcol}}, \mathrm{A_{\wcol\bcol}}$ are the sets of interval, noncrossing, and all partitions (ignoring the colors), and thus are known to come from the trivially two-faced Boolean, free and tensor universal product, respectively. Swapping the colors does not change these sets of partitions.
  \item The set $\mathrm{NC_\wcol I_\bcol}$ is the set of noncrossing-interval partitions, which originates from free-Boolean independence \cite{Liu19}. Swapping the colors leads to the set $\mathrm{I_\wcol NC_\bcol}$.  
  \item The set $\mathrm{A_\wcol I_\bcol}$ comes from tensor-Boolean independence \cite{GHU23}. Swapping the colors leads to the set $\mathrm{I_\wcol A_\bcol}$.
  \item The set $\mathrm{biNC}$ is the set of binoncrossing partitions, it comes from bifree independence \cite{CNS15,Voi14}. Swapping the colors does not change the set.
  \end{itemize}
  We are left with the sets of pure crossing and pure noncrossing partitions and with the sets of noncrossing-arbitrary and arbitrary-noncrossing partitions, where again by swapping the colors it is enough to deal with the noncrossing-arbitrary ones. All properties are easily verified.

  The theorem now follows from the fact that each admissible set is uniquely determined by which basic two-block partitions have nonzero coefficients, and from the implications in \Cref{cor:basic-partitions}. 
\end{proof}

\begin{corollary}\label{cor:classification}
  Let $\odot$ be a positive symmetric 2-faced universal product. Then the admissible set of partitions
  \[\Pi_\odot:=\{\pi\in\mathcal P:\alpha_\pi\neq 0\}\]
  is one of the 12 given in \Cref{def:admissible-sets-concrete}. Furthermore:
  \begin{itemize}
  \item If $\Pi_\odot\in\{\mathrm{NC_\wcol A_\bcol},\mathrm{A_\wcol NC_\bcol},\mathrm{pNC_{\wcol\bcol}},\mathrm{pC_{\wcol\bcol}}\}$, then the highest coefficients of $\odot$ are given by the indicator function of $\Pi_\odot$, and $\odot$ does not coincide with any of the positive symmetric two-faced universal product given in \cite[Propositions 5.13 and 6.19]{GHU23}.
  \item In all other cases, $\odot$ does coincide with one of the positive symmetric two-faced universal product given in \cite[Propositions 5.13 and 6.19]{GHU23}; more concretely,
    \begin{itemize}
    \item if $\Pi_\odot=\mathrm{A}_{\wcol\bcol}$, then $\odot=\uttprod{}{\zeta}{}{\phantom{\zeta}}$ with $\zeta=\overline{\xi_{\wcol\bcol}}=\overline{\nu_{\wcol\bcol}}$ is a deformed tensor product,
    \item if $\Pi_\odot=\mathrm{NC}_{\wcol\bcol}$, then $\odot=\uffprod{}{\zeta}{}{\phantom{\zeta}}$ with $\zeta=\overline{\nu_{\wcol\bcol}}$ is a deformed free product,
    \item if $\Pi_\odot=\mathrm{biNC}_{\wcol\bcol}$, then $\odot=\ubfprod{}{\zeta}{}{\phantom{\zeta}}$ with $\zeta=\overline{\nu_{\wcol\bcol}}$ is a deformed bifree product,
    \item if $\Pi_\odot=\mathrm{I}_{\wcol}\mathrm{A}_{\bcol}$, then $\odot=\twofacedprod{\bool}{\otimes}$ is the Boolean-tensor product, 
    \item if $\Pi_\odot=\mathrm{A}_{\wcol}\mathrm{I}_{\bcol}$, then $\odot=\twofacedprod{\otimes}{\bool}$ is the tensor-Boolean product,
    \item if $\Pi_\odot=\mathrm{I}_{\wcol}\mathrm{NC}_{\bcol}$, then $\odot=\twofacedprod{\bool}{\Asterisk}$ is the Boolean-free product, 
    \item if $\Pi_\odot=\mathrm{NC}_{\wcol}\mathrm{I}_{\bcol}$, then $\odot=\twofacedprod{\Asterisk}{\bool}$ is the free-Boolean product,
    \item if $\Pi_\odot=\mathrm{I}_{\wcol\bcol}$, then $\odot=\boolboolprod$ is the Boolean product.
    \end{itemize}
  \end{itemize}
\end{corollary}

\begin{proof}
  If $\odot$ is a positive symmetric universal product, then its highest coefficients form an admissible family of weights. If all the basic coefficients are 0 or 1, the family must be given by the indicator function of one of the admissible sets of partitions and all except the mentioned four are identified as positive products in \cite{GHU23}:
  \begin{itemize}
  \item $\mathrm{A}_{\wcol\bcol}$ corresponds to the tensor product
  \item $\mathrm{NC}_{\wcol\bcol}$ corresponds to the free product
  \item $\mathrm{biNC}_{\wcol\bcol}$ corresponds to the bifree product
  \item $\mathrm{I}_{\wcol}\mathrm{A}_{\bcol}$ and $\mathrm{A}_{\wcol}\mathrm{I}_{\bcol}$  corresponds to the Boolean-tensor and tensor-Boolean product, respectively
  \item $\mathrm{I}_{\wcol}\mathrm{NC}_{\bcol}$ and $\mathrm{NC}_{\wcol}\mathrm{I}_{\bcol}$  corresponds to the Boolean-free and free-Boolean product, respectively
  \item $\mathrm{I}_{\wcol\bcol}$ corresponds to the Boolean product
  \end{itemize}
  If one of the basic coefficients is not 0 or 1,  \Cref{lem:relations-for-basic-coefficients}  leaves only three possibilities, in each of which the universal product has been found to be positive in \cite{GHU23}:
  \begin{itemize}
  \item $\nu_{\wcol\bcol}=\xi_{\wcol\bcol}=q\in\mathbb T\setminus\{1\}$, in this case all other basic coefficients are forced to be equal to 1; by comparison of the basic coefficients, the corresponding universal product is the deformed tensor product with $\zeta=\overline q$,
  \item $\nu_{\wcol\bcol}=q\in\mathbb T\setminus\{1\}$, $\xi_{\wcol\bcol}=0$; in this case, the product must coincide with the deformed free product with $\zeta=\overline q$,
  \item $\nu_{\wcol\bcol}=0$, $\xi_{\wcol\bcol}=q\in\mathbb T\setminus\{1\}$; in this case, the product must coincide with the deformed bifree product with $\zeta=\overline q$.\qedhere
  \end{itemize}
\end{proof}

\begin{remark}\label{rem:trace-preserving}
  A remarkable property of freeness is that the free product of traces is again a trace. We cannot expect such a behaviour for any non-trivial multi-faced independence. Indeed, this would force the highest coefficients to be invariant under cyclic permutations, and since we may change the color of the first leg, we could change the color of every leg without changing the coefficient.   
\end{remark}

\begin{remark}\label{rem:convolution-R2}
  Bifreeness allows to define a convolution for probability measures on $\mathbb R^2$. This comes from the fact that for bifree pairs $(a_1^{\wcol},a_2^{\bcol}), (b_1^{\wcol},b_2^{\bcol})$ one always has commutativity of $a_1^{\wcol}$ with $b_2^{\bcol}$ and of $a_2^{\bcol}$ with $b_1^{\wcol}$. Consequently, $a_1^{\wcol}+b_1^{\wcol}$ commutes with $a_2^{\bcol}+b_2^{\bcol}$ whenever $a_1^{\wcol}$, $a_2^{\bcol}$ commute and $b_1^{\wcol},b_2^{\bcol}$ commute. If independent variables in different faces commute, one must have $\xi_{\wcol\bcol}=1$, which is only the case for tensor and bifree independence. 
\end{remark}

\begin{remark}
  There are other interesting symmetric two-faced universal products which are not positive, for example the bi-Boolean product. It seems very well possible to do a classification under slightly relaxed conditions, only assuming that one is allowed to change the color of the first leg and not assuming any mirror symmetry (recall that we used changing the color on both sides to show that highest coefficients for all interval partitions are 1). However, it is not clear how to motivate those properties when one does not aim for positivity. For the construction of a universal product in the algebraic sense (see \Cref{sec:reconstruction}), \Cref{it:highest-coeff:change-color,it:highest-coeff:symmetry} of \Cref{thm:properties-of-highest-coeffs} are not necessary at all. 
\end{remark}

\section{Moment-cumulant relations}\label{sec:moment-cumulant-relations}

A key tool in the proofs of the subsequent sections are cumulants. In this section, we adapt the theory of cumulants developed in \cite{MaSc17} to our special case of symmetric multi-faced independences.

\begin{observation}\label{obs:moment-cumulant-relation}
  Let $\alpha$ be a family of weights such that $\alpha_\pi$ is invertible for every one-block partition. For every family of \emph{moments}, $m=(m_{\mathbf{f}})_{{\mathbf{f}}\in \mathcal F^*}\in\mathbb C^{\mathcal F^*}$, there is a unique family of \emph{$\alpha$-cumulants}, $c=(c_{\mathbf{f}})_{{\mathbf{f}}\in \mathcal F^*}\in\mathbb C^{\mathcal F^*}$ such that 

  \begin{align}\label{eq:sf-moment-cumulant-relation}
    m_{\mathbf{f}}=\sum_{\pi\in \mathcal P_<({\mathbf{f}})} \frac{1}{|\pi|!} \alpha_\pi \prod_{\beta\in\pi} c_{|\beta|};
  \end{align}
  indeed, existence and uniqueness of the $c_{\mathbf{f}}$ follows by a standard induction argument. Obviously, the cumulants also determine the moments.

  If $\alpha$ is invariant under permutation of blocks, then the formula simplifies to
  \begin{align}\label{eq:sf-symmetric-moment-cumulant-relation}
    m_{\mathbf{f}}=\sum_{\pi\in \mathcal P({\mathbf{f}})} \alpha_\pi \prod_{\beta\in\pi} c_{|\beta|}.
  \end{align}

  There is no problem extending formulas \eqref{eq:sf-moment-cumulant-relation} and \eqref{eq:sf-symmetric-moment-cumulant-relation} to a multivariate situation. To this end, we think of the (multivariate) moments and cumulants as linear functionals $m,c\colon A\to\mathbb C$, where $A=\mathbb C\langle x_{i}^\bsquare:\bsquare\in\mathcal F, i\in I^\bsquare\rangle$ is a multi-faced polynomial algebra with (possibly) several indeterminates $x_{i}^\bsquare$, $i\in I^{\bsquare}$, for each face $\bsquare \in\mathcal F$. For a monomial $X=x_{\mathbf i(1)}^{{\mathbf{f}}(1)}\cdots x_{\mathbf i(n)}^{{\mathbf{f}}(n)}$ and a subset $\beta=\{\ell_1<\ldots < \ell_r\}\subset[n]$, let $X\restriction \beta$ denote the monomial $x_{\mathbf i(\ell_1)}^{{\mathbf{f}}(\ell_1)}\cdots x_{\mathbf i(\ell_n)}^{{\mathbf{f}}(\ell_n)}$. Cumulants are then defined by the relations
  \begin{align}
    m(X)&=\sum_{\pi\in\mathcal P_<({\mathbf{f}})}\frac{1}{|\pi|!}\alpha_\pi\prod_{\beta\in \pi} c(X\restriction \beta),\label{eq:non-symmetric_moment-cumulant_relation}\\
    m(X)&=\sum_{\pi\in\mathcal P({\mathbf{f}})}\alpha_\pi\prod_{\beta\in \pi} c(X\restriction \beta),\label{eq:symmetric_moment-cumulant_relation}
  \end{align}
  respectively. In case each $I^{\bsquare}$ is a one-element set, writing $x^{\bsquare}$ for the indeterminates, formulas \eqref{eq:non-symmetric_moment-cumulant_relation} and \eqref{eq:symmetric_moment-cumulant_relation} are recovered by setting $m_{\mathbf f}:=m(x^{\mathbf f(1)}\cdots x^{\mathbf f(n)})$ and $c_{\mathbf f}:=c(x^{\mathbf f(1)}\cdots x^{\mathbf f(n)})$ for $\mathbf f\in\mathcal F^n$.
\end{observation}

\begin{definition}
  An \emph{algebraic probability space} is a pair $(\mathcal A,\Phi)$, where $\mathcal A$ is an algebra and $\Phi\colon \mathcal A\to\mathbb C$ is a linear functional.
\end{definition}

\begin{definition}\label{def:moments-cumulants}
  Let $(\mathcal A,\Phi)$ be an algebraic probability space and $\alpha=(\alpha_\pi)_{\pi\in\mathcal P_{(<)}}$ a family of weights on (ordered) multi-faced partitions. For a family $\mathbf a=(a_i^\bsquare:\bsquare\in\mathcal F,i\in I^\bsquare)\subset \mathcal A$, put $j_{\mathbf a}\colon \mathbb C\langle x_{i}^\bsquare:\bsquare\in\mathcal F, i\in I^\bsquare\rangle\to \mathcal A, x_i^\bsquare\mapsto a_i^\bsquare$. We define its \emph{moments} by
  \[m_{\mathbf a}(X):=\Phi(j_{\mathbf a}(X))\]
  and its \emph{$\alpha$-cumulants} $c_{\mathbf a}(X)$ according to the moment-cumulant relations \eqref{eq:non-symmetric_moment-cumulant_relation} or \eqref{eq:symmetric_moment-cumulant_relation}, respectively.
\end{definition}

\begin{definition}\label{def:exp-log}
  Fix monic weights $(\alpha_\pi)_{\pi\in\mathcal P}$.
  Let $V=\bigoplus_{\bsquare\in\mathcal F}V^{\bsquare}$ be a vector space with a direct sum decomposition into subspaces according to the faces. Recall that $T_0(V)=\bigoplus _{n\in\mathbb N} V^{\otimes n}=\bigsqcup_{\bsquare\in\mathcal F} T_0(V^{\bsquare})$ denotes the (non-unital) free algebra over $V$ and $T(V):= \bigoplus _{n\in\mathbb N_0}V^{\otimes n}=\mathbb C 1\oplus T_0(V)$ its unitization, the free unital algebra over $V$. On the dual space $T_0(V)'=\{\varphi\colon T_0(V)\to\mathbb C \text{ linear}\}$ we define for $x_i\in V^{{\mathbf{f}}(i)}$
  \[\exp_\alpha(\psi)(x_1\cdots x_n):=\sum_{\pi\in \mathcal P({\mathbf{f}}) } \alpha_\pi\psi^{\otimes |\pi|}(x_\pi)\]
  where for $\pi=\{\beta_1,\ldots, \beta_n\}$ we put $x_{\beta_k}:=\mathop{\overrightarrow{\prod}}_{i\in \beta_k} x_i$ and $x_\pi:=x_{\beta_1}\otimes\cdots\otimes x_{\beta_n}$. Then $\exp_\alpha$ is a bijection. We denote the inverse simply as $\log_\alpha$, which can be calculated recursively,
  \[\log_{\alpha}(\varphi)(x_1\cdots x_n)=\varphi(x_1\cdots x_n)
    -\sum_{\substack{\pi\in \mathcal P({\mathbf{f}})\\|\pi|>1}} \alpha_\pi\log_{\alpha}(\varphi)^{\otimes |\pi|}(x_\pi).\]
  Note that often $\exp_\alpha$ and $\log_\alpha$ are interpreted as bijections between linear functionals on $T(V)$ vanishing on 1 and unital linear functionals on $T(V)$ by extending the linear functionals from $T_0(V)$ to $T(V)$ accordingly (i.e.\ $\psi$ and $\log_\alpha(\varphi)$ are extended by annihilating the unit, while $\exp_\alpha(\psi)$ and $\varphi$ are extended as unital maps).

  We use the following conventions.
  \begin{itemize}
  \item If the weights $\alpha$ come from a universal product $\odot$, we write $\exp_\odot:=\exp_\alpha$ and $\log_\odot:=\log_\alpha$. 
  \item If $(x_i^{\bsquare})_{i\in I^{\bsquare}}$ form a basis of $V^{\bsquare}$, we identify $T(V)$ and $T_0(V)$ with the noncommutative (unital or non-unital) polynomial algebras $\mathbb C\langle{x_i^{\bsquare}:i\in I^{\bsquare},\bsquare\in\mathcal F}\rangle$ and $\mathbb C\langle{x_i^{\bsquare}:i\in I^{\bsquare},\bsquare\in\mathcal F}\rangle_0$, respectively.
  \end{itemize}
\end{definition}

\begin{definition}\label{def:hat}
  Let $A$ be a multi-faced algebra and $\varphi\colon A\to\mathbb C$ a linear functional. We define $\hat A:=T_0\Bigl(\bigoplus_{\bsquare\in\mathcal F }A^\bsquare\Bigr)$ and $\hat\varphi:=\varphi\circ\mu$, where $\mu\colon T_0\Bigl(\bigoplus_{\bsquare\in\mathcal F }A^\bsquare\Bigr)\to A$ is the canonical homomorphism.
\end{definition}

\begin{observation}\label{obs:log-cumulants}
  Let $\alpha$ be monic weights. Let furthermore $\varphi\colon A\to\mathbb C$ be a linear functional on a multi-faced algebra $A$ and $\mathbf a=(a_i^{\bsquare}\in A^{\bsquare}:i\in I^{\bsquare},\bsquare\in\mathcal F)$ a family of elements. With the notations from the previous definitions, for $X=x_{i_1}^{\bsquare_1}\otimes\cdots\otimes x_{i_n}^{\bsquare_n}$, it holds that 
  \[\hat\varphi(a_{i_1}^{\bsquare_1}\otimes\cdots\otimes a_{i_n}^{\bsquare_n})=m_{\mathbf a}(x_{i_1}^{\bsquare_1}\otimes\cdots\otimes x_{i_n}^{\bsquare_n}) \quad\text{and}\quad\log_\alpha\hat\varphi (a_{i_1}^{\bsquare_1}\otimes\cdots\otimes a_{i_n}^{\bsquare_n})=c_{\mathbf a}(x_{i_1}^{\bsquare_1}\otimes\cdots\otimes x_{i_n}^{\bsquare_n}).\]
\end{observation}

\begin{observation}\label{obs:exp-log-universal}
  Let $h\colon B\to A$ be an $\mathcal F$-faced homomorphism between $\mathcal F$-faced algebras $B,A$ and define $\hat h \colon \hat B\to \hat A$ as the unique algebra homomorphism with $\hat h(b) =h(b)$ for all $b\in B^{\bsquare}$, $\bsquare\in\mathcal F$. Automatically, $\hat h$ is an $\mathcal F$-faced homomorphism and fulfills $\mu_A\circ\hat h=h\circ \mu_B$. For $\varphi\colon A\to\mathbb C$ a linear functional, it follows that
  \[\widehat{\varphi\circ h}= \varphi\circ h \circ \mu_B= \varphi\circ\mu_A\circ \hat h=\hat\varphi \circ\hat h. \]
  Therefore, given monic weights $(\alpha_\pi)_{\pi\in\mathcal P}$, one finds that
  \[\exp_\alpha(\widehat {\varphi\circ h})=(\exp_\alpha \hat\varphi)\circ\hat h,\quad \log_\alpha(\widehat {\varphi\circ h})=(\log_\alpha \hat\varphi)\circ\hat h.\]
\end{observation}

\begin{theorem}[{Adjusted and simplified from  \cite[Th.~7.2]{MaSc17}}]\label{thm:unique-highest-coefficients}\ \\
  A positive and symmetric universal product is uniquely determined by its highest coefficients. More precisely, for $a=a_1\cdots a_n$ with $a_\ell\in A_{\mathbf b(\ell)}^{\mathbf f(\ell)}$ so that $a_1\otimes\cdots \otimes a_n\in T_0\Bigl(\bigoplus_{\kappa\in[2],\bsquare\in\mathcal F }A_\kappa^\bsquare\Bigr)=\hat A_1\sqcup\hat A_2$,  
  \[\varphi_1\odot\varphi_2(a_1\cdots a_n)=\exp_\odot \bigl(\log_{\odot}(\hat\varphi_1)\oplus \log_{\odot}(\hat\varphi_2)\bigr)(a_1\otimes\cdots \otimes a_n);\]
  here we use the  direct sum as a shorthand notation for the corresponding linear functional on $T_0\Bigl(\bigoplus_{\kappa\in[2],\bsquare\in\mathcal F }A_i^\bsquare\Bigr)=\hat A_1\sqcup\hat A_2$ as described by \Cref{eq:direct-sum-extended}. 
\end{theorem}

\begin{proof}
  We only explain why this is a special case of \cite[Th.~7.2]{MaSc17} and refer the reader to \cite[Theorems 2.4.12 and 2.5.13]{Var21} for a detailed discussion.  Since $\odot$ is positive, their are no wrong-ordered highest coefficients. In the symmetric case, the exponential and logarithm map used in \cite{MaSc17} coincide with the maps of \Cref{def:exp-log} and are therefore determined by the highest coefficients. Since $\odot$ is symmetric, the second ingredient which is in general needed to determine the universal product, namely the \emph{$n$th order cumulant Lie algebra}, is trivial for all $n$. 
\end{proof}

\section{Reconstruction of universal products from highest coefficients }\label{sec:reconstruction}

In this section we prove that every admissible family leads to a unique universal product. In particular, we can associate universal products with the admissible sets  $\mathrm{NC_\wcol A_\bcol},\mathrm{A_\wcol NC_\bcol},\mathrm{pNC_{\wcol\bcol}},\mathrm{pC_{\wcol\bcol}}$. However, it remains an open problem at the moment to decide whether or not those universal products are positive.

\begin{lemma}\label{lem:cumulants-of-product}
  Suppose that the weights $(\alpha_\pi)_{\pi \in \mathcal{P}} $ are admissible.  Fix a family of elements $\mathbf a=(a_\ell^\bsquare:\bsquare\in\mathcal F,\ell\in I^\bsquare)\subset \mathcal A$ in an algebraic probability space $(\mathcal A,\Phi)$ such that $[n]$ is the disjoint union of the $I^{\bsquare}$. Put ${\mathbf{f}}(\ell):=\bsquare$ if $\ell\in I^{\bsquare}$ and assume that ${\mathbf{f}}(i)={\mathbf{f}}(i+1)=\wsquare\in\mathcal F$ for a certain index $i\in[n]$. 
  We define a modified family $\tilde{\mathbf a}=(a_\ell^\bsquare:\bsquare\in\mathcal F,\ell\in \tilde{I}^\bsquare)$ where $\tilde{I}^{\wsquare}:=I^{\wsquare}\setminus\{i,i+1\}\cup\{\{i,i+1\}\}$ and $a_{\{i,i+1\}}^{\wsquare}:=a_i^{\wsquare}a_{i+1}^{\wsquare}$.
  For $X:=x_1^{{\mathbf{f}}(1)}\cdots x_n^{{\mathbf{f}}(n)}$, 
  $\tilde X:=x_1^{{\mathbf{f}}(i)}\cdots x_{i-1}^{{\mathbf{f}}(i-1)}x_{\{i,i+1\}}^{\wsquare} x_{i+2}^{{\mathbf{f}}(i+2)}\cdots x_n^{{\mathbf{f}}(n)}$
  the moments and cumulants according to \Cref{def:moments-cumulants}  fulfill
  $m_{\tilde{\mathbf a}}(\tilde X)=m_{\mathbf a}(X)$ and
  \[c_{\tilde{\mathbf a}}(\tilde X)=c_{\mathbf a}(X)+\sum_{\substack{\sigma=\{\beta_1,\beta_2\}\in\mathcal P({\mathbf{f}})\\i\in \beta_1\neq \beta_2\ni{i+1}}} \alpha_{\sigma}c_{\mathbf a}(X\restriction \beta_1)c_{\mathbf a}(X\restriction \beta_2).\]
\end{lemma}
\begin{proof}
  The claimed equality for the moments is obvious.
  The claim for the cumulants is proved by induction on $n$. For $n=2$, i.e.\ $X=x_1^{\wsquare}x_2^{\wsquare}$, we have
  \[c_{\tilde{\mathbf a}}(x_{\{1,2\}}^{\wsquare})=m_{\tilde{\mathbf a}}(x_{\{1,2\}}^{\wsquare})=m_{{\mathbf a}}(x_1^{\wsquare}x_2^{\wsquare})=c_{\mathbf a}(x_1^{\wsquare}x_2^{\wsquare})+ \alpha\left(\sqpartxy\right)c_{\mathbf a}(x_1^{\wsquare})c_{\mathbf a}(x_2^{\wsquare}).\]

  For general $n$, we can use the moment-cumulant relations for $m_{\mathbf a}(X)=m_{\tilde{\mathbf a}}(\tilde X)$ and obtain
  \allowdisplaybreaks\begin{align*}
    m_{\mathbf a}(X)&=\sum_{\pi\in \mathcal P({\mathbf{f}})} \alpha_\pi \prod_{\beta\in\pi} c_{\mathbf a}(X\restriction \beta)\\
    &\begin{multlined}[t][.9\linewidth]=\sum_{\substack{\pi\in \mathcal P({\mathbf{f}})\\i,i+1\in \hat \beta\in \pi}}\alpha_\pi c_{\mathbf a}(X\restriction\hat \beta)\prod_{\beta\in\pi\setminus \{\hat \beta\}}c_{\mathbf a}(X\restriction \beta)\\ 
    + \sum_{\substack{\rho\in\mathcal P(\mathbf f)\\\beta_1,\beta_2\in \rho\\i\in \beta_1\neq \beta_2\ni i+1}}\alpha_\rho c_{\mathbf a}(X\restriction \beta_1) c_{\mathbf a}(X\restriction \beta_2) \prod_{\beta\in\rho\setminus{\{\beta_1,\beta_2\}}}c_{\mathbf a}(X\restriction \beta)\end{multlined} \\
    &
      \begin{multlined}[t][.9\linewidth]
        =\sum_{\substack{\pi\in\mathcal P(\mathbf f)\\i,i+1\in \hat \beta\in \pi}} \alpha_\pi\Biggl(c_{\mathbf a}(X\restriction \hat \beta)  + \sum _{\substack{\{\beta_1,\beta_2\}\in\mathcal P(\hat \beta)\\i\in \beta_1\neq \beta_2\ni i+1}} \alpha_{\{\beta_1,\beta_2\}} c_{\mathbf a}(X\restriction \beta_1) c_{\mathbf a}(X\restriction \beta_2) \Biggr)\\\cdot\prod_{\beta\in\pi\setminus \{\hat \beta\}}c_{\mathbf a}(X\restriction \beta)
      \end{multlined}
    \\
           &
             \begin{multlined}[t][.9\linewidth]
               =c_{\mathbf a}(X)+\sum_{\substack{\{\beta_1,\beta_2\}\in\mathcal P(\mathbf f)\\i\in \beta_1\neq \beta_2\ni i+1}}\alpha_{\{\beta_1,\beta_2\}}c_{\mathbf a}(X\restriction \beta_1)c_{\mathbf a}(X\restriction \beta_2)\\
                +\sum_{\substack{1_{\mathbf{f}}\neq\pi\in\mathcal P({\mathbf{f}})\\i,i+1\in \hat \beta\in \pi}} \alpha_\pi\Biggl(c_{\mathbf a}(X\restriction \hat \beta)  + \sum _{\substack{\{\beta_1,\beta_2\}\in\mathcal P(\hat \beta)\\i\in \beta_1,i+1\in \beta_2}} \alpha_{\{\beta_1,\beta_2\}} c_{\mathbf a}(X\restriction \beta_1) c_{\mathbf a}(X\restriction \beta_2) \Biggr)\\\cdot \prod_{\beta\in\pi\setminus \{\hat \beta\}}c_{\mathbf a}(X\restriction \beta)
             \end{multlined}
  \end{align*}
  where we used $\alpha_\rho=\alpha_\pi\alpha_{\{\beta_1,\beta_2\}}$ for $\pi=\rho_{\beta_1\smile \beta_2}$. On the other hand, with $\tilde{\mathbf{f}}\in \mathcal F^{[n]/(i\sim i+1)}$, $\tilde{\mathbf{f}}(\{i,i+1\}):={\mathbf{f}}(i)={\mathbf{f}}(i+1)$ and $\tilde{\mathbf{f}}(\ell):={\mathbf{f}}(\ell)$ for $\ell\neq \{i,i+1\}$,
  \begin{align*}
    m_{\tilde{\mathbf a}}(\tilde X)
        &=\sum_{\sigma\in \mathcal P(\tilde {\mathbf{f}})}\alpha_\sigma \prod_{\beta\in\sigma}c_{\tilde{\mathbf a}}(\tilde X\restriction \beta)\\
        &=\sum_{\substack{\sigma\in\mathcal P(\tilde{\mathbf{f}})\\ \{i,i+1\}\in \tilde \beta}} \alpha_{\sigma} c_{\tilde{\mathbf a}}(\tilde X\restriction \tilde \beta)\prod_{\beta\in\sigma\setminus\{\tilde \beta\}}c_{{\mathbf a}}(X\restriction \beta)\\ 
        &
        \begin{multlined}[t][.9\textwidth]                                                                                  =c_{\tilde{\mathbf a}}(\tilde X) +    \sum_{\substack{1_{\tilde{\mathbf f}}\neq\sigma\in\mathcal P(\tilde{\mathbf{f}})\\ \{i,i+1\}\in \tilde \beta}} \alpha_{\sigma} c_{\tilde{\mathbf a}}(\tilde X\restriction \tilde \beta)\prod_{\beta\in\sigma\setminus\{\tilde \beta\}}c_{\mathbf a}(X\restriction \beta).
        \end{multlined}
  \end{align*}
  Recall that there is a canonical bijection between partitions $\sigma\in \mathcal P(\tilde{\mathbf f})$ and partitions $\pi\in\mathcal P(\mathbf f)$ with $i,i+1$ in the same block $\hat \beta\in\pi$. Also, the highest coefficients $\alpha_\sigma$ and $\alpha_\pi$ agree under this bijection by \Cref{thm:properties-of-highest-coeffs} \ref{it:highest-coeff:red}. Using the induction hypothesis on $c_{\mathbf a}(X\restriction\hat \beta)$ finishes the proof. 
\end{proof}

\begin{theorem}\label{thm:reconstruction}
  Suppose that the weights $(\alpha_\pi)_{\pi \in \mathcal{P}} $ are admissible. Then there exists a unique symmetric universal product with highest coefficients $(\alpha_\pi)_{\pi \in \mathcal{P}}$.
\end{theorem}

\begin{proof}
  The uniqueness statement is proved in \cite[Th.~7.2]{MaSc17}, see \Cref{thm:unique-highest-coefficients}.

  Let $\varphi_\kappa\colon A_\kappa\to\mathbb C$ be linear functionals on 2-faced algebras ($\kappa\in\{1,2\}$). Recall \Cref{def:exp-log} of $\exp_\alpha$ and $\log_\alpha$ and \Cref{def:hat}, which sets the notation for lifting $\varphi_k$ to linear functionals $\hat\varphi_\kappa=\varphi_\kappa\circ\mu_\kappa$ on the tensor algebras $T_0(\bigoplus_{\bsquare\in \mathcal F}A_\kappa^{\bsquare})$. 
  We simply write $\exp:=\exp_\alpha$ and $\log:=\log_\alpha$ in the following. 
  We define
  \begin{align}
    \varphi_1\mathbin{\widetilde\odot}\varphi_2:=\exp\bigl(\log(\hat\varphi_1)\oplus \log(\hat\varphi_2)\bigr)\in T_0\left(\bigoplus_{\bsquare\in\mathcal F} (A_1^\bsquare\oplus A_2^\bsquare)\right)'.
  \end{align}
  The main task is now to prove that $\varphi_1\mathbin{\widetilde\odot}\varphi_2$ vanishes on the ideal $\mathcal I:=\operatorname{ker}(\mu_1\sqcup \mu_2)$ in
  \[T_0\left(\bigoplus_{\bsquare\in\mathcal F} (A_1^\bsquare\oplus A_2^\bsquare)\right)=\bigsqcup_{\bsquare\in\mathcal F}\Bigl(T_0(A_1^\bsquare)\sqcup T_0(A_2^\bsquare)\Bigr)\]
  (i.e.\ the ideal generated by the relations $a\otimes b=ab$ for $a,b\in A_\kappa^\bsquare$), so that $\varphi_1\mathbin{\widetilde\odot}\varphi_2$ descends to a functional $\varphi_1\odot\varphi_2$ with $(\varphi_1\odot\varphi_2)\circ(\mu_1\sqcup\mu_2)=(\varphi_1\mathbin{\widetilde\odot}\varphi_2)$ on the quotient
  \[A_1\sqcup A_2={T_0\left(\bigoplus_{\bsquare\in\mathcal F} (A_1^\bsquare\oplus A_2^\bsquare)\right)}/{\mathcal I}.\]
  Let $\mathbf s=\mathbf f\times \mathbf b\in([2]\times\mathcal F)^n$ with $\mathbf s(i)=\mathbf s(i+1) = (j,\bsquare)$ for some $i\in[n-1],j\in[2],\bsquare\in\mathcal F$. Let $a_1\cdots a_n\in A^{\mathbf s}\subset A_1\sqcup A_2$ with $a_\ell\in A_{\mathbf b(\ell)}^{{\mathbf{f}}(\ell)}$, in particular, $a_{i}$ and $a_{i+1}$ lie in the same direct summand $A_j^{\bsquare}$ of the free product $A_1\sqcup A_2$. Define $\tilde{\mathbf{f}}\in \mathcal F^{[n]/(i\sim i+1)}$, $\tilde{\mathbf{f}}(\{i,i+1\})={\mathbf{f}}(i)={\mathbf{f}}(i+1)$ and $\tilde{\mathbf{f}}(\ell)={\mathbf{f}}(\ell)$ for $\ell\neq \{i,i+1\}$. Analogously, we define $\tilde{\mathbf{b}}$ and $\tilde{\mathbf{s}}$.  With $\mathbf a:=(a_1,\ldots, a_n)$, $X:=x_1\cdots x_n\in \mathbb C\langle x_1,\ldots, x_n\rangle$, $\tilde{\mathbf a}:=(a_1,\ldots,a_ia_{i+1},\ldots a_n)$, and $\tilde X:=x_1\cdots x_{\{i,i+1\}} x_n\in \mathbb C\langle x_1,\ldots,x_{\{i,i+1\}}, x_n\rangle$, we have
  \begin{itemize}
  \item for $\beta\subset \{1,\ldots, i,i+1,\ldots, n\}$ with $a_\ell\in A_\kappa$ for all $\ell\in \beta $
    \[\log\hat\varphi_\kappa (\mathbf a \restriction \beta)=c_{\mathbf a}(X\restriction \beta),\]
  \item for $\beta\subset \{1,\ldots, \{i,i+1\},\ldots, n\}$ with $a_\ell\in A_\kappa$ for all $\ell\in \beta $
    \[\log\hat\varphi_\kappa (\tilde{\mathbf a} \restriction \beta)=c_{\tilde{\mathbf a}}(\tilde{X}\restriction \beta).\]
  \end{itemize}
  Let us say that a partition $\pi$ is \emph{adapted to $\mathbf b$}, and write $\pi\prec \mathbf b$, if $\mathbf b$ is constant on blocks of $\pi$ (this is the first condition of $\pi$ being adapted to $\mathbf s$). Note that $(\log\hat\varphi_1\oplus\log\hat\varphi_2)^{\otimes |\pi|}(a_\pi)=0$ when $\pi$ is not adapted to $\mathbf b$; indeed, this follows directly from the way we identify the direct sum of linear functionals with a linear functional on the tensor algebra in \Cref{eq:direct-sum-extended}.
  With this in hand, we calculate
  \begin{multline}
    {\exp(\log \hat\varphi_1\oplus\log\hat\varphi_2)(a_1\otimes\cdots\otimes a_ia_{i+1}\otimes\cdots \otimes a_n)}\\=\sum_{\pi\in\mathcal P(\tilde{\mathbf{f}})} \alpha_\pi (\log\hat\varphi_1\oplus\log\hat\varphi_2)^{\otimes |\pi|}(a_\pi)
    =\sum_{\substack{\pi\prec\tilde{\mathbf{b}}\\\{i,i+1\}\in\tilde \beta\in\pi}} \alpha_\pi c_{\tilde{\mathbf a}}(\tilde X\restriction \tilde \beta)\prod_{\beta\in\pi\setminus\{\tilde \beta\}}c_{\tilde{\mathbf a}}(\tilde X\restriction \beta).\label{eq:exp(log+log)-identified}
  \end{multline}
  On the other hand, if the two legs $i$ and $i+1$ are not identified, then there are partitions adapted to $\mathbf b$
  for which $i,i+1$ lie in the same block as well as ones for which $i,i+1$ lie in different blocks. This leads to
  \begin{align}
    \MoveEqLeft\exp(\log \hat\varphi_1\oplus \log\hat\varphi_2)(a_1\otimes\cdots\otimes a_i\otimes a_{i+1}\otimes\cdots \otimes a_n)\label{eq:exp(log+log)-not_identified}\\\notag
    &=\sum_{\pi\in\mathcal P({\mathbf{f}})} \alpha_\pi (\log\hat\varphi_1\oplus\log\hat\varphi_2)^{\otimes |\pi|}(a_\pi)\\\notag
    &=\sum_{\substack{\pi\prec{\mathbf{b}}\\i,i+1\in\hat \beta\in\pi}}\alpha_\pi c_{{\mathbf a}}( X\restriction \hat \beta)\prod_{\beta\in\pi\setminus\{\hat \beta\}}  c_{{\mathbf a}}( X\restriction \beta)\\\notag
    &\quad+\sum_{\substack{\sigma\prec{\mathbf{b}}\\ \beta_1,\beta_2\in\sigma \\ i\in \beta_1\neq \beta_2\ni i+1}}
    \alpha_\sigma c_{{\mathbf a}}( X\restriction \beta_1)c_{{\mathbf a}}( X\restriction \beta_2)\prod_{\beta\in\sigma\setminus\{\beta_1,\beta_2\}}  c_{{\mathbf a}}( X\restriction \beta)\\\notag
    &=\sum_{\substack{\pi\prec{\mathbf{b}}\\i,i+1\in\hat \beta\in\pi}} \alpha_\pi\left(c_{{\mathbf a}}( X\restriction \hat \beta)+\sum_{\beta_1\dot\cup \beta_2=\hat \beta} \alpha_{\{\beta_1,\beta_2\}}c_{{\mathbf a}}( X\restriction \beta_1)c_{{\mathbf a}}( X\restriction \beta_2)\right)\prod_{\beta\in \pi\setminus\{\hat \beta\}} c_{{\mathbf a}}( X\restriction \beta),
  \end{align}
  using $\alpha_\pi\alpha_{\{\beta_1,\beta_2\}}=\alpha_\sigma$ when $\sigma=\pi\setminus\{\hat \beta\}\cup\{\beta_1,\beta_2\}$, i.e.\ $\pi=\sigma_{\beta_1\smile \beta_2}$. 
  The two expressions derived in \eqref{eq:exp(log+log)-identified} and \eqref{eq:exp(log+log)-not_identified} agree by \Cref{lem:cumulants-of-product} and, therefore, we have a well-defined map $\varphi_1\odot\varphi_2(a_1\cdots a_n)=\varphi_1\odot\varphi_2(a_1\otimes\cdots \otimes a_n+\mathcal I):=\varphi_1\mathbin{\widetilde\odot}\varphi_2(a_1\otimes\cdots \otimes a_n)$.

  Let us verify that $\odot$ is indeed a symmetric universal product. To prove universality, recall \Cref{obs:exp-log-universal}. Let $h_\kappa\colon B_\kappa\to A_\kappa$ be $\mathcal F$-faced algebra homomorphisms and $\varphi_\kappa\colon A_\kappa\to\mathbb C$ linear functionals. Then, for $b=b_1\cdots b_n\in B_1\sqcup B_2$ and $\hat b=b_1\otimes\cdots \otimes b_n\in \hat B_1\sqcup \hat B_2$,
    \begin{align*}
      (\varphi_1\circ h_1)\odot(\varphi_2\circ h_2)(b)
      &= \exp(\log(\widehat{\varphi_1\circ h_1})\oplus \log (\widehat{\varphi_2\circ h_2}))(\hat b)\\
      &=\exp((\log\hat\varphi_1 \circ \hat h_1)\oplus (\log\hat\varphi_2\circ\hat h_2))(\hat b)\\
      &= \exp((\log\hat\varphi_1\oplus \log\hat\varphi_2)\circ (\hat h_1\sqcup \hat h_2))(\hat b)\\
      &= \exp(\log\hat\varphi_1\oplus \log\hat\varphi_2)\circ (\hat h_1\sqcup \hat h_2)(\hat b)\\
      &= ((\varphi_1\odot \varphi_2)\circ(h_1\sqcup h_2))(b).
    \end{align*}
    Symmetry and unitality are immediate for $\widetilde\odot$ and therefore descend to $\odot$. To prove associativity is slightly more involved. We write $b\in A_1\sqcup A_2\sqcup A_3$ as $(\mu_1\sqcup \mu_2\sqcup \mu_3)(\hat b)$ with \[\hat b\in T_0\left(\bigoplus_{\bsquare\in\mathcal F} (A_1^\bsquare\oplus A_2^\bsquare\oplus A_3^\bsquare)\right)=\bigsqcup_{\bsquare\in\mathcal F}T_0(A_1^\bsquare)\sqcup T_0(A_2^\bsquare)\sqcup T_0(A_3^\bsquare)\]
    and claim that \[((\varphi_1\odot\varphi_2)\odot\varphi_3)(b)=(\varphi_1\odot(\varphi_2\odot\varphi_3))(b)=\exp(\log\hat\varphi_1\oplus\log\hat\varphi_2\oplus\log\hat\varphi_3)(\hat b).\]
    The crucial observation is that $\log\widehat{\varphi_1\odot\varphi_2} = (\log\hat\varphi_1\oplus\log\hat\varphi_2)\circ\lambda_{12}$ with  \[\lambda_{12}\colon T_0\left(\bigoplus_{\bsquare\in\mathcal F}(A_1\sqcup A_2)^{\bsquare}\right)=\bigsqcup_{\bsquare\in\mathcal F} T_0(A_1^{\bsquare}\sqcup A_2^{\bsquare})\to T_0\left(\bigoplus_{\bsquare\in\mathcal F} (A_1^{\bsquare}\oplus A_2^{\bsquare})\right)= \bigsqcup_{\bsquare\in\mathcal F} \left(T_0(A_1^{\bsquare})\sqcup T_0(A_2^{\bsquare})\right) \]
    the unique algebra homomorphism extending the canonical embeddings $A_1^{\bsquare}\sqcup A_2^{\bsquare}\hookrightarrow T_0(A_1^{\bsquare})\sqcup T_0(A_2^{\bsquare})$. Indeed, $\widehat{\varphi_1\odot\varphi_2}=(\varphi_1\odot\varphi_2)\circ \mu_{12}$ for the canonical map
    \[\mu_{12}\colon  T_0\left(\bigoplus_{\bsquare\in\mathcal F}(A_1\sqcup A_2)^{\bsquare}\right)\to A_1\sqcup A_2,\]
    which factorizes as $\mu_{12}=(\mu_1\sqcup \mu_2)\circ\lambda_{12}$. From $(\varphi_1\odot\varphi_2)\circ(\mu_1\sqcup\mu_2)=\varphi_1\mathbin{\widetilde\odot}\varphi_2$, we conclude that \[\widehat{\varphi_1\odot\varphi_2}=(\varphi_1\odot\varphi_2)\circ\mu_{12} = (\varphi_1\odot\varphi_2)\circ(\mu_1\sqcup\mu_2)\circ\lambda_{12} = (\varphi_1\mathbin{\widetilde\odot}\varphi_2)\circ\lambda_{12}  .\] From \Cref{def:exp-log} it is obvious that $\exp(\psi\circ\lambda_{12})=\exp(\psi)\circ\lambda_{12}$ for all $\psi\in T_0\left(\bigoplus_{\bsquare\in\mathcal F}(A_1^{\bsquare}\oplus A_2^{\bsquare})\right)'$, therefore $\log\widehat{\varphi_1\odot\varphi_2} = (\log\hat\varphi_1\oplus\log\hat\varphi_2)\circ\lambda_{12}$ as claimed. The rest is easy:
    \begin{align*} ((\varphi_1\odot\varphi_2)\odot\varphi_3)(b)&= \exp(\log(\widehat{\varphi_1\odot\varphi_2})\oplus \log \hat\varphi_3)(\hat b)\\
      &= \exp(\log(\hat\varphi_1)\oplus\log(\hat\varphi_2)\oplus \log (\hat\varphi_3))((\lambda_{12}\sqcup\mathrm{id})(\hat b))\\
      &=\exp(\log(\hat\varphi_1)\oplus \log(\hat\varphi_2)\oplus \log (\hat\varphi_3))(\hat b);
    \end{align*}
    note that $\lambda_{12}$ is actually a projection onto a subalgebra, so we can safely identify $\hat b$ with the corresponding element in the domain of $\lambda_{12}\sqcup\mathrm{id}$ instead of introducing yet another symbol for its preimage.
    The other direction, i.e.\ $(\varphi_1\odot(\varphi_2\odot\varphi_3))(b)=\exp(\log(\hat\varphi_1)\oplus \log(\hat\varphi_2)\oplus \log (\hat\varphi_3))(\hat b)$, follows by symmetry.

  To check that the highest coefficients of $\odot$ are indeed given by $\alpha$, it is enough to consider products of two functionals $\varphi_1,\varphi_2$. For $a=a_1\cdots a_n\in A^{\mathbf s}$,  $\mathbf s=\mathbf b\times\mathbf f\in([2]\times\mathcal F)^{*}$, and $\sigma\in\mathcal P(\mathbf f)$ the partition with blocks $\beta_\kappa=\{\ell:\mathbf b(\ell)=\kappa\}$, we find
  \begin{align*}
    \MoveEqLeft\left.\frac{\partial^2}{\partial t_1\partial t_2}(t_1\varphi_1)\odot(t_2\varphi_2) (a_1\cdots a_n)\right\vert_{\mathbf t=0}
    =\sum_{\pi\in\mathcal P({\mathbf{f}})} \left.\frac{\partial^2}{\partial t_1\partial t_2}\alpha_\pi\cdot (\log t_1 \hat\varphi_1\oplus \log t_2\hat\varphi_2)^{\otimes |\pi|}(a_\pi)\right\vert_{\mathbf t=0}\\
    &=\left.\frac{\partial^2}{\partial t_1\partial t_2}\alpha_\sigma\cdot\log t_1\hat\varphi_1 (a_{\beta_1}) \log t_2\hat\varphi_2 (a_{\beta_2})\right\vert_{\mathbf t=0}\\
    &= \alpha_\sigma\cdot\varphi_1 (a_{\beta_1}) \varphi_2 (a_{\beta_2})
  \end{align*}
  as needed.\footnote{The notation $a_\beta=\mathop{\overrightarrow{\prod}}_{i\in \beta}a_i$, $a_\pi=\bigotimes_{\beta\in \pi}a_\beta$ refers to $a_1\otimes\cdots \otimes a_n$, but note that by well-definedness the choice of decomposition of $a_1\cdots a_n\in A_1\sqcup A_2$ as a tensor in $T_0(\bigoplus_{\bsquare\in\mathcal F} A_1^{\bsquare}\oplus A_2^{\bsquare})$ does not influence the result!}
\end{proof}

The formula to compute mixed moments can be considerably simplified in the special case where the the highest coefficients are only 0 or 1. 

\begin{definition}\label{def:combinatorial}
  We say that a symmetric universal product is \emph{combinatorial with partition set $\Pi$} if its highest coefficients are all either 0 or 1 and $\Pi=\{\pi\in\mathcal P:\alpha_{\pi}=1\}$.
\end{definition}

\begin{theorem}\label{thm:combinatorial-mixed-moments}
  Let $\odot$ be a combinatorial universal product with admissible partition set $\Pi$ one of the 12 sets of \Cref{thm:classification-two-faced-admissible-sets} (in particular, $\mathcal F=\{\wcol, \bcol\}$ a two element set). Furthermore, let $\varphi_\kappa$ be a linear functional on a multi-faced algebra $A_\kappa$ ($\kappa\in[k]$), and $a=a_1\cdots a_n\in A^{\mathbf s}$ for $\mathbf s=\mathbf b\times \mathbf f\in([k]\times\mathcal F)^{n}$. Denote 
  \begin{itemize}
  \item $\pi\in\mathcal P(\mathbf f)$ the multi-faced partition with blocks $\beta_\kappa:=\{i:\mathbf b(i)=\kappa\}$ (whenever non-empty),
  \item $\Pi_{\leq \pi}:=\{\sigma\in \Pi:\sigma\leq \pi\}$ the set of refinements of $\pi$ inside $\Pi\cap\mathcal P(\mathbf f)$,
  \item $S$ the set of maximal elements of $\Pi_{\leq\pi}$ (i.e.\ coarsest refinements of $\pi$ inside $\Pi$),
  \item $\wedge R$ is the maximal common refinement of partitions in $R\subset \Pi\cap \mathcal P(\mathbf f)$, $\wedge\emptyset:=1_{\mathbf f}$,
  \item $\Phi:=\varphi_1\odot\cdots\odot\varphi_k$,
  \item $\hat\Phi$ the lift of $\Phi$ to $\displaystyle T_0\Bigl(\bigoplus\limits_{\ell\in[k], \bsquare\in\mathcal F} A_\ell^{\bsquare}\Bigr)$.
  \end{itemize}
  Then
  \[\Phi(a)=\sum_{\emptyset\neq R\subseteq S} (-1)^{\# R-1} \hat\Phi^{\otimes |\wedge R|}(a_{\wedge R}).\]
\end{theorem}
\begin{proof}
  Put $\Psi:=\log \hat\Phi=\log \hat\varphi_1\oplus\cdots\oplus\log\hat\varphi_k$. The key observation is that a refinement $\sigma$ of a partition $\rho\in \Pi$ belongs to $\Pi$ if and only if $\sigma\restriction \beta\in \Pi$ for all blocks $\beta\in \rho$; this can be easily seen for each of the 12 admissible sets of partitions individually. Using the moment cumulant formula on each block of $\wedge R$ and the observation on refinements just made, we find
    \begin{align}\label{eq:combinatorial-proof}
      \sum_{R\subseteq S} (-1)^{\#R} \hat\Phi^{\otimes |\wedge R|}(a_{\wedge R})= \sum_{R\subseteq S} \sum_{\substack{\sigma\leq {{\wedge} R}\\\sigma\in\Pi}} (-1)^{\#R} \Psi^{\otimes |\sigma|}(a_{\sigma})
    \end{align}
    (equality of the summands for $R=\emptyset$ will be discussed below.)
    Now, the same partition $\sigma\in\Pi$ can of course be a refinement of $\wedge R$ for different $R\subseteq S$. Denote $T(\sigma):=\{\rho\in S:\sigma\leq \rho\}$ and $n(\sigma):=\# T(\sigma)$. Then $\sigma\leq \wedge R$ if and only if $R\subseteq T(\sigma)$, and for every $k\in\{0,1,\ldots,n(\sigma)\}$ there are $\binom{n(\sigma)}{k}$ many such $R$ with $\# R=k$. If $n(\sigma)=0$, i.e.\ if $\sigma$ is not a refinement of $\pi$, then $\Psi^{\otimes|\sigma|}(a_\sigma)=0$ because mixed cumulants vanish. This leads to    
    \begin{align*}
       \textrm{RHS of \eqref{eq:combinatorial-proof}}=\sum_{\sigma\in\Pi_{\leq\pi}}\underbrace{{\sum_{k=0}^{n(\sigma)} (-1)^k \binom{n(\sigma)}{k}}}_{=0} \Psi^{\otimes |\sigma|}(a_\sigma)=0.
    \end{align*}
    Recall that we defined $\wedge\emptyset:=1_{\mathbf f}$, so that
    \[\Phi(a)=\hat\Phi(a_{1_{\mathbf f}})=\sum_{\sigma\in \mathcal P(\mathbf f)\cap \Pi}\Psi^{\otimes|\sigma|}(a_\sigma)=\sum_{\substack{\sigma\leq 1_{\mathbf f}\\\sigma\in\Pi}}\Psi^{\otimes|\sigma|}(a_\sigma);\]
    this confirms that the choice is consistent with \Cref{eq:combinatorial-proof}, and it also shows that the statement of the theorem is equivalent to $\text{LHS of \eqref{eq:combinatorial-proof}}=0$.
  \end{proof}

  \begin{example}
    Let $\odot$ be the universal product associated with $\mathrm{NC_{\wcol}A_{\bcol}}$. Then $\partxyoxoxyo$ has set of coarsest refinements $S=\{\partxyoxozyo,\partxyoxoxzo\}$ in $\mathrm{NC_{\wcol}A_{\bcol}}$ with $\wedge S=\partxyoxozwo$, leading to
    \[\varphi\odot\psi(a_1^{\wcol}b_1^{\bcol} a_2^{\bcol}a_3^{\wcol}b_2^{\bcol})
      =\varphi(a_1^{\wcol}a_2^{\bcol})\varphi(a_3^{\wcol})\psi(b_1^{\bcol}b_2^{\bcol})
      +\varphi(a_1^{\wcol}a_2^{\bcol}a_3^{\wcol})\psi(b_1^{\bcol})\psi(b_2^{\bcol})
      - \varphi(a_1^{\wcol}a_2^{\bcol})\varphi(a_3^{\wcol})\psi(b_1^{\bcol})\psi(b_2^{\bcol})\]
    for all $\varphi\in A',\psi\in B'$, $a_i^{\gbullet}\in A^{\gbullet}$ ($i=1,2,3$), and $b_j^{\gbullet}\in B^{\gbullet}$ ($j=1,2$). 
  \end{example}

\section{Unit preserving universal products}\label{unit-preserving}

In \cite{DAGSV22}, Diaz-Aguilera, Gaxiola, Santos, and Vargas characterize when the moment cumulant relation associated with weights on partitions leads to independent constants, finding this to be the case if and only if the weights do not change when removing or inserting a singleton from or to the partition.
Manzel and Sch{\"u}rmann discuss in \cite[Rem.\ 3.1]{MaSc17} the relation between universal products in the category of multi-faced algebras and in the category of multi-faced unital algebras and observe that while a product for the unital category always gives rise to a product for the non-unital category, the other way round requires a condition, namely that the universal product \emph{respects the units} or is  \emph{unit preserving} as we prefer to write in this article.
In this section we briefly review universal products in the category of multi-faced unital algebras, define what exactly it means to be unit preserving, generalize the definition of \emph{singleton inductive weights} to the multi-faced setting, and finally characterize unit preserving symmetric universal product as those whose highest coefficients are singleton inductive.

In the category of unital algebras with unital algebra homomorphisms, the coproduct is given by the \emph{unital free product}, which can be constructed from the non-unital free product as \[\mathcal A_1\unitalsqcup \mathcal A_2:=(\mathcal A_1\sqcup \mathcal A_2)/{\langle 1_{\mathcal A_1}-1_{\mathcal A_2}\rangle};\]
here $\langle \cdot\rangle$ denotes the generated two-sided ideal.

\begin{definition}\ 
  \begin{itemize}[beginpenalty=10000]
  \item A \emph{multi-faced unital algebra} is a unital algebra $\mathcal \mathcal A$ with unital subalgebras $\mathcal A^{\bsquare}$, $\bsquare\in\mathcal F$, such that the canonical unital algebra homomorphism $\bigunitalsqcup_{\bsquare\in\mathcal F}\mathcal A^{\bsquare}\to \mathcal A$ is an isomorphism, in which case we write $\mathcal A=\bigunitalsqcup_{\bsquare\in\mathcal F}\mathcal A^{\bsquare}$.
  \item A \emph{multi-faced unital algebra homomorphism} is a unital algebra homomorphism which maps face into face.
  \item The unital free product of multi-faced unital algebras $\mathcal A_1,\mathcal A_2$ is a multi-faced unital algebra with $(\mathcal A_1\unitalsqcup \mathcal A_2)^{\bsquare}:=\mathcal A_1^{\bsquare}\unitalsqcup \mathcal A_2^{\bsquare}$.
  \item A linear functional $\phi\colon \mathcal A\to\mathbb C$ on a multi-faced unital algebra is \emph{unital} if $\phi(1_\mathcal A)=1$. 
  \end{itemize}
\end{definition}

Multi-faced unital algebras with multi-faced unital algebra homomorphisms form a category, in which $\unitalsqcup$ is a coproduct. One can adapt \Cref{def:universal-product} to the unital situation and obtains the following.

\begin{definition}
  A \emph{universal product in the category of multi-faced unital algebras} is a binary product operation for unital linear functionals on multi-faced unital algebras which associates with unital functionals $\phi_1,\phi_2$ on multi-faced unital algebras $\mathcal A_1,\mathcal A_2$, respectively, a unital functional $\phi_1\odot\phi_2$ on $\mathcal A_1\unitalsqcup \mathcal A_2$ such that
  \begin{itemize}
  \item $(\phi_1\circ h_1)\odot(\phi_2\circ h_2)=(\phi_1\odot\phi_2)\circ (h_1\unitalsqcup h_2)$ for all multi-faced unital algebra homomorphisms $h_i\colon\mathcal B_i\to \mathcal A_i$ (universality)
  \item $(\phi_1\odot\phi_2)\odot\phi_3 = \phi_1\odot(\phi_2\odot\phi_3)$ (associativity)
  \item $(\phi_1\odot\phi_2)\restriction \mathcal A_i=\phi_i$ (restriction property).
  \end{itemize}
\end{definition}

As Manzel and Sch{\"u}rmann noticed in \cite[Rem.\ 3.1]{MaSc17}, every universal product $\widetilde\odot$ in the category of multi-faced unital algebras gives rise to a universal product in the sense of \Cref{def:universal-product}, simply putting
\[\varphi_1\odot\varphi_2:=\tilde\varphi_1\mathbin{\widetilde\odot}\tilde\varphi_2\restriction A_1\sqcup A_2\subset \tilde A_1\unitalsqcup\tilde A_2\]
where $\tilde A_i$ denotes the unitization of a multi-faced algebra and $\tilde\varphi_i$ the unital extension of a linear functional.

Conversely, if a universal product $\odot$ in the non-unital case is given, one would like to define
\begin{align}\label{eq:unital-odot-well-defined}
\phi_1\mathbin{\widetilde\odot}\phi_2(p(a)):=\varphi_1\odot \varphi_2(a)
\end{align}
with the following conventions: 
\begin{itemize}
\item $A_\kappa :=\bigsqcup_{\bsquare\in\mathcal F}\mathcal A_\kappa ^{\bsquare}$, so that  
    \(\mathcal A_\kappa \cong A_\kappa /\mathcal I_{\mathcal A_\kappa }\) with \(\mathcal I_{\mathcal A_\kappa }:=\langle 1^{\wsquare} - 1^{\bsquare}:{\wsquare},{\bsquare}\in\mathcal F\rangle\subset A_\kappa\),
  \item $p_{\mathcal A_\kappa }\colon A_\kappa \to\mathcal A_\kappa $ denotes the canonical homomorphism,
  \item $\varphi_\kappa :=\phi_\kappa \circ p_{\mathcal A_\kappa }\colon A_\kappa \to \mathbb C$, i.e.\   \(\varphi_\kappa (a):=\phi_\kappa (a+\mathcal I_{\mathcal A_\kappa })\),
  \item $p\colon A_1\sqcup A_2\to \mathcal A_1\unitalsqcup \mathcal A_2$ denotes the canonical homomorphism.
\end{itemize}

\begin{definition}\label{def:unit-preserving}
  A universal product is \emph{unit preserving} (or \emph{respects units}) if, whenever  $A_1,A_2$ are multi-faced algebras with each $A_i^{\bsquare}$ unital and $\varphi_i$ a linear functional on $A_i$ which vanishes on the ideal $\langle 1_i^{\wsquare}-1_i^{\bsquare}:{\wsquare},{\bsquare}\in\mathcal F\rangle\subset A_i$ and such that $\varphi_i\restriction A_i^{\bsquare}$ is unital for every $\bsquare\in\mathcal F$, then $\varphi_1\odot\varphi_2$ vanishes on the ideal $\langle 1_{i}^{\wsquare} - 1_{j}^{\bsquare}:i,j\in[2],{\wsquare},{\bsquare}\in\mathcal F\rangle\subset A_1\sqcup A_2$ and $\varphi_1\odot\varphi_2\restriction A_i^{\bsquare}$ is unital for every $i\in[2]$, $\bsquare\in\mathcal F$. 
\end{definition}

\begin{remark}
  A multi-faced universal product is unit preserving if and only if \eqref{eq:unital-odot-well-defined} is well-defined, in which case it yields a universal product in the category of multi-faced unital algebras \cite[Rem.\ 3.1]{MaSc17}. Since Manzel and Sch{\"u}rmann do not give a definition of ``respecting units'', let us briefly check that \Cref{def:unit-preserving} captures what they mean.

  Assume that $\odot$ is unit preserving. The $\varphi_i=\phi_i\circ p_{\mathcal A_i}$ in \eqref{eq:unital-odot-well-defined} are linear functionals on $A_i$, vanish on $\operatorname{ker}p_{\mathcal A_i}=\mathcal I_{\mathcal A_i}=\langle 1_i^{\wsquare}-1_i^{\bsquare}:{\wsquare},{\bsquare}\in\mathcal F\rangle\subset A_i$ and fulfill $\varphi_i(1_i^{\bsquare})=\phi_i(1_i)=1$.  Therefore, we may conclude that $\varphi_1\odot\varphi_2$ vanishes on the ideal $\langle 1_{i}^{\wsquare} - 1_{j}^{\bsquare}:i,j\in[2],{\wsquare},{\bsquare}\in\mathcal F\rangle\subset A_1\sqcup A_2$, which coincides with the kernel of the canonical homomorphism $p\colon A_1\sqcup A_2\to\mathcal A_1\unitalsqcup \mathcal A_2$. This means that there is a well-defined linear functional $\phi_1\mathbin{\widetilde\odot}\phi_2$ with $\varphi_1\odot\varphi_2=(\phi_1\mathbin{\widetilde\odot}\phi_2)\circ p$. This functional is also unital because $\phi_1\mathbin{\widetilde\odot}\phi_2(1)=\phi_1\mathbin{\widetilde\odot}\phi_2(p(1_i^{\bsquare}))=\varphi_i(1_i^{\bsquare})=1$.

We leave the rest of the simple, but notationally cumbersome proof of the claim (in particular universality and associativity of $\widetilde\odot$) to the interested reader. 
\end{remark}

In the following we will need often remove a singleton block $\beta=\{s\}$ from a partition $\pi\ni \beta$. While consistent use of notation would dictate to write $\pi\setminus \{\beta\}=\pi\setminus \{\{s\}\}$, we will prefer to write $\pi\setminus \{s\}$ for better legibility.

\begin{definition}[multi-faced version of {\cite[Def.\ 3.2]{DAGSV22}}]
  A family of weights $(\alpha_\pi)_{\pi\in\mathcal P}$ is \emph{singleton inductive} if $\alpha_\pi=\alpha_{\pi\setminus \{s\}}$ for every singleton block $\{s\}\in\pi$.
\end{definition}

\begin{lemma}[multi-faced version of {\cite[Th.\ 3.2]{DAGSV22}}]\label{lem:singleton-inductive-cumulants}
  Let $\alpha$ be monic, singleton inductive weights. Suppose that $A=\bigsqcup_{\bsquare\in\mathcal F}A^{\bsquare}$ is a multi-faced algebra such that each face $A^{\bsquare}$ is unital (with unit $1^{\bsquare}$) and that $\varphi$ fulfills $\varphi(1^{\bsquare})=1$ and $\varphi$ vanishes on the ideal $\langle 1^{\wsquare}-1^{\bsquare}:{\wsquare},{\bsquare}\in\mathcal F\rangle\subset A$. Then \[\log_\alpha \hat\varphi (a_1\otimes\cdots \otimes a_n)=0 \text{ whenever $n>1$ and $a_s=1^{\bsquare}$ for some $s\in [n],\bsquare\in\mathcal F$;}\] here $\hat\varphi$ is the lift of $\varphi$ to a $T_0\left(\bigoplus_{\bsquare\in \mathcal F}A^{\bsquare}\right)$. 
\end{lemma}

\begin{proof}
  We prove the claim by induction. For $n=2$ and arbitrary ${\wsquare},{\bsquare}\in\mathcal F$,
  \[\log_\alpha \hat\varphi(1^{\wsquare}\otimes a^{\bsquare})=\varphi(1^{\wsquare}a^{\bsquare})-\varphi(1^{\wsquare}) \varphi(a^{\bsquare})=\varphi(1^{\bsquare}a^{\bsquare})-\varphi(a^{\bsquare})= 0\]
  and analogously $\log_\alpha \hat\varphi(a^{\wsquare}\otimes 1^{\bsquare})=0$. Now assume the statement holds for all $1<m<n$ and consider $a=a_1\otimes\cdots \otimes a_n$ with $a_i\in A^{\mathbf f(i)}$, $a_s=1^{\mathbf f(s)}$. Note that $\varphi(a_1\cdots a_n)=\varphi(a_1\cdots \check a_s\cdots a_n)$ (here $\check a_s$ means omission of the factor) and $\log_\alpha\hat\varphi(a_s)=\log_\alpha\hat\varphi(1^{\bsquare})=\varphi(1^{\bsquare})=1$. We find
  \begin{align*}
    \MoveEqLeft\log_\alpha \hat\varphi(a_1\otimes \cdots \otimes a_n)=\varphi(a_1\cdots a_n)-\sum_{\pi\in \mathcal P(\mathbf f)\setminus \{1_{\mathbf f}\}} \alpha_\pi (\log_\alpha \hat\varphi)^{\otimes |\pi|}(a_\pi)\\
    &=\varphi(a_1\cdots  a_n)-\sum_{\{s\}\in\pi\in \mathcal P(\mathbf f)} \alpha_\pi (\log_\alpha \hat\varphi)^{\otimes |\pi|}(a_\pi)-\underbrace{\sum_{\{s\}\notin\pi\in \mathcal P(\mathbf f)\setminus \{1_{\mathbf f}\}} \alpha_\pi (\log_\alpha \hat\varphi)^{\otimes |\pi|}(a_\pi)}_{=0 \text{ by induction hypothesis}}\\
    &=\varphi(a_1\cdots \check a_s\cdots  a_n)-\sum_{\{s\}\in\pi\in \mathcal P(\mathbf f)} \alpha_{\pi\setminus\{s\}} (\log_\alpha \hat\varphi)^{\otimes |\pi|-1}(a_{\pi\setminus\{s\}})\log_\alpha\hat\varphi(a_s)\\
    &=\varphi(a_1\cdots \check a_s\cdots  a_n)-\sum_{\sigma\in \mathcal P(\mathbf f\restriction [n]\setminus \{s\})} \alpha_{\sigma} (\log_\alpha \hat\varphi)^{\otimes |\sigma|}(a_{\sigma})=0,
  \end{align*}
  where we used that the weights are singleton inductive as well as the moment cumulant relation for $a_1\cdots \check a_s\cdots  a_n$.
\end{proof}

\begin{theorem}\label{prop:unit-preserving}
  For a multi-faced positive symmetric universal product $\odot$, the following are equivalent.
  \begin{enumerate}
  \item $\odot$ is unit preserving,
  \item $\nu_\bsquare=1$ for all $\bsquare\in\mathcal F$,
  \item the highest coefficients of $\odot$ are singleton inductive.
  \end{enumerate}
\end{theorem}

\begin{proof}
  Let $\odot$ be unit preserving. To calculate $\nu_\bsquare=\alpha\bigl(\sqnestbullet\bigr)$, we can assume that also the extremal legs are $\bsquare$-legs. We can therefore ignore the faces and calculate, as in the single-faced case, \[\varphi_1\odot\varphi_2(aba')=\nu_\bsquare \cdot\varphi_1(aa')\varphi_2(b)+\gamma\cdot \varphi_1(a)\varphi_1(a')\varphi_2(b)\] for all $a,a'\in A_1^{\bsquare}, b\in A_2^{\bsquare}$, with some universal constant $\gamma\in\mathbb C$. Suppose that $\varphi_1,\varphi_2$ are as in \Cref{def:unit-preserving} and furthermore $b=1_2^{\bsquare},\varphi_1(a)=\varphi_1(a')=0,\varphi_1(aa')=1$, then \[\varphi_1\odot\varphi_2(a1_2^{\bsquare}a')=\varphi_1\odot\varphi_2(a1_1^{\bsquare}a')=\varphi_1\odot\varphi_2(aa')=\varphi_1(aa')=1\]
  because $\odot$ preserves units. We also have $\varphi_2(b)=\varphi_2(1_2^{\bsquare})=1$. Putting everything together, $\nu_\bsquare=1$.   

  A simple induction on the number of blocks shows that $\alpha_\pi=\nu_{\mathbf \bsquare}\cdot\alpha_{\pi\setminus\{s\}}$ whenever $\pi\in \mathcal P$ has a singleton block $\{s\}\in\pi$ of color $\bsquare$. Therefore, $\nu_\bsquare=1$ for all $\bsquare\in\mathcal F$ implies that the highest coefficients are singleton inductive.
  
  Now assume that the highest coefficients of a positive symmetric universal product are singleton inductive.
  Let $A_1,A_2$ be multi-faced algebras with unital faces, $\mathbf s=\mathbf b\times \mathbf f\in([2]\times \mathcal F)^{n}$, $a_\ell\in A_{\mathbf b(\ell)}^{\mathbf f(\ell)}$ for $\ell\in[n]$, $\hat a=a_1\otimes \cdots \otimes a_n$, $a=a_1\cdots a_n$ and $a_s=1_{i}^{\bsquare}$. Then, with $\log:=\log_\odot$,
  \begin{align*}
    \varphi_1\odot\varphi_2(a)=\sum_{\pi\in\mathcal P(\mathbf f)} \alpha_\pi(\log\hat\varphi_1\oplus\log\hat\varphi_2)^{\otimes |\pi|}(\hat a_\pi)
    = \sum_{\{s\}\in\pi\in\mathcal P(\mathbf f)} \alpha_\pi(\log\hat\varphi_1\oplus\log\hat\varphi_2)^{\otimes |\pi|}(\hat a_\pi)
  \end{align*}
  by \Cref{lem:singleton-inductive-cumulants}. Because $\alpha$ is singleton inductive, $\alpha_\pi=\alpha_{\pi\setminus\{s\}}$. Also, for any $\pi$ with $\{s\}\in\pi$, we have \[(\log\hat\varphi_1\oplus\log\hat\varphi_2)^{\otimes |\pi|}(\hat a_\pi)=(\log\hat\varphi_1\oplus\log\hat\varphi_2)^{\otimes |\pi\setminus\{s\}|}(\hat a_{\pi\setminus\{s\}})\underbrace{\log \hat\varphi_{\mathbf b(s)}(a_s)}_{=1}.\]
  Therefore, $\varphi_1\odot\varphi_2(a_1\cdots a_n)=\varphi_1\odot\varphi_2(a_1\cdots \check{a}_s\cdots a_n)$. This calculation works for any $i\in[2]$ and any $\bsquare\in\mathcal F$, so the statement follows. 
\end{proof}

\begin{corollary}
  A 2-faced positive symmetric universal product is unit preserving if and only if its associated set of partitions contains $\mathrm{pNC_{\wcol\bcol}}$.
\end{corollary}

\section{Summary and outlook}
\label{sec:outlook}

We found conditions on weights that are necessarily satisfied by the highest coefficients of a positive two-faced universal product. In the symmetric case, we showed that weights which fulfill these conditions are always the highest coefficients of a uniquely determined universal product. We could also determine all families of weights which fulfill these conditions, thereby providing a list of candidates for positive symmetric universal products.   

We hope that the methods developed in this work will eventually lead to a complete classification of positive multi-faced universal products. To that end, the following problems will have to be overcome:

\begin{itemize}
\item Prove or disprove positivity of the ``exceptional cases'' which do not admit a representation on free or tensor product.  
\item Extend the classification of admissible weights to more than two faces. 
\item Extend the classification of admissible weights to the non-symmetric case.
\item Extend the reconstruction theorem to the non-symmetric case. This might be significantly more difficult because the cumulants have to be combined using the Campbell-Baker-Hausdorff formula instead of just the direct sum.  
\end{itemize}

\appendix

\section{Comparison with \texorpdfstring{\cite{Var21}}{[Var21]}}
\label{sec:comparison}

Most ideas behind the proofs in \Cref{sec:necessary-conditions,sec:classification,sec:reconstruction} go more or less back to \cite{Var21}.  A crucial difference between this article and the exposition in \cite{Var21} is that our main results are consistently formulated and proved for families of weights on partitions, while Var\v{s}o often works with sets of partitions instead, which means that in \cite{Var21} several results are only proved in the \emph{combinatorial} case in the sense of \Cref{def:combinatorial}. The weight-based approach often helped us to streamline proofs. Another difference is that we decided to focus on \emph{positive} universal products here. 

In the following we give a more detailed comparison of the results.
\begin{itemize}
\item \Cref{thm:properties-of-highest-coeffs} (iv) is basically \cite[Corollary 5.2.6]{Var21}. The remaining claims of \Cref{thm:properties-of-highest-coeffs} generalize \cite[Theorem 5.2.17]{Var21} to possibly non-symmetric universal products. Because we put more emphasis on positive products, for ease of reading, we only formulated \Cref{thm:properties-of-highest-coeffs} for positive products while Var\v{s}o formulates his results more generally for products with the ``right ordered monomials property'', i.e.\ those products for which the conclusion of \Cref{thm:MS:highest-coefficients} holds; however, we mention in the proof where exactly the positivity condition is used and where the right-ordered monomials property is enough.
\item \Cref{lem:coincide-on-2block} is closely related to \cite[Theorem 5.2.20]{Var21} (since admissible families of weights are not defined in \cite{Var21}, the statement is formulated for families of highest coefficients of certain universal products). \Cref{lem:relations-for-basic-coefficients} has overlap with \cite[Lemma 5.2.23]{Var21}; however, from \Cref{lem:relations-for-basic-coefficients} (3) it follows that all coefficients have absolute value in $\{0,1\}$, which goes beyond what was found in \cite{Var21}. Regarding the main classification results,
  \Cref{thm:classification-two-faced-admissible-sets} corresponds to \cite[Theorem 4.2.44]{Var21} and \Cref{cor:classification} strengthens \cite[Remark 5.2.28]{Var21}.
\item \Cref{obs:operations-for-partitions} draws the connection between the admissible sets of partitions as defined from admissible weights in \Cref{def:admissible,def:admissible-partitions} and Var\v{s}o's ($m$-colored) universal classes of partitions \cite[Definition 3.4.9]{Var21}. The only difference is that admissible sets are assumed to contain the interval partitions, while a universal class of partitions is also allowed to consist of the 1-block partitions alone. 
\item Our reconstruction theorem, \Cref{thm:reconstruction}, also covers universal products with non-0-1 highest coefficients, in contrast to \cite[Theorem 3.4.32]{Var21}.  The crucial \Cref{lem:cumulants-of-product} corresponds to \cite[Lemma 3.4.24]{Var21} (formulated and proved for admissible weights instead of universal classes of partitions).
\end{itemize}

\Cref{thm:combinatorial-mixed-moments} and the results of \Cref{unit-preserving} have no counterpart in \cite{Var21}.

\section*{Acknowledgements}

We are grateful to Michael Sch{\"u}rmann for numerous fruitful discussions in the course of this research. MG thanks Moritz Weber and Roland Speicher for stimulating discussions and atmosphere during his stay in Saarbr{\"u}cken for the focus semester on quantum information Autumn 2022. We thank both anonymous referees for their comments, which helped us improve the quality of this article. We truly appreciate one of the referees extraordinary detailed reading and in particular his thoughtful suggestions regarding notation.

%%% Local Variables:
%%% mode: latex
%%% TeX-master: "main"
%%% ispell-local-dictionary: "english"
%%% End:

\bibliographystyle{myalphaurl-sortbyauthor}
\bibliography{biblio}
\setlength{\parindent}{0pt}
\end{document}